\documentclass[10pt]{amsart}
\usepackage{amssymb}

\usepackage{graphicx}
\usepackage{xcolor}
\usepackage{stmaryrd}

\newcommand{\res}{\upharpoonright}

\theoremstyle{plain}
\newtheorem{theorem}{Theorem}[section]

\newtheorem{lemma}[theorem]{Lemma}

\theoremstyle{remark}

\newtheorem{claim}{Claim}
\newtheorem*{claim*}{Claim}

\begin{document}

\title[Generic transformations and closed groups they generate]{Generic measure 
preserving transformations\\ and the closed groups they generate}

\author{S{\l}awomir Solecki}

\thanks{The author was partially supported by NSF grant DMS-1954069.}

\address{Department of Mathematics\\
Cornell University\\
Ithaca, NY 14853}

\email{ssolecki@cornell.edu} 



\begin{abstract}
We show that, for a generic measure preserving transformation $T$, the closed group generated by $T$ is not isomorphic 
to the topological 
group $L^0(\lambda, {\mathbb T})$ of all Lebesgue measurable functions from $[0,1]$ to $\mathbb T$ (taken with pointwise multiplication 
and the topology of convergence in measure).
This result answers a question of Glasner and Weiss. 
The main step in the proof consists of showing that Koopman representations of ergodic boolean actions of 
$L^0(\lambda, {\mathbb T})$ possess a non-trivial spectral property 
not shared by all unitary representations of $L^0(\lambda, {\mathbb T})$. 
The main tool underlying our arguments is a theorem on the form of unitary representations of 
$L^0(\lambda, {\mathbb T})$ from our earlier work. 
\end{abstract}

\maketitle

\setcounter{tocdepth}{1}

\section{Introduction} 

\subsection{Two groups and some notational conventions}\label{Su:twno} 

Let $\gamma$ be an atomless Borel probability measure on a standard Borel space $X$. 
By 
\[
{\rm Aut}(\gamma)
\]
we denote the topological group of all (measure equivalence classes of) 
measurable, measure preserving bijections of $X$. The group operation in ${\rm Aut}(\gamma)$ is composition 
and the topology is the weak topology, that is, 
the weakest topology making the functions 
\[
{\rm Aut}(\gamma) \ni T\to \gamma\big(T(A)\big)\in {\mathbb R}
\]
continuous, for Borel sets $A\subseteq X$. 
For more information on the group ${\rm Aut}(\gamma)$, the reader may consult \cite[Sections~1 and 2]{Kec2}. 

Let $\nu$ be a finite Borel measure on a standard Borel space $Y$. As usual, $\mathbb T$ stands for 
the group of all complex numbers of unit length taken with multiplication. By 
\[
L^0(\nu, {\mathbb T})
\]
we denote the topological group of all (measure equivalence classes of) measurable functions from 
$Y$ to $\mathbb T$. The group operation on $L^0(\nu, {\mathbb T})$ is pointwise multiplication and the topology 
is the topology of convergence in measure. 

The topologies on both ${\rm Aut}(\gamma)$ and $L^0(\nu, {\mathbb T})$ are separable 
and completely metrizable, that is, both these groups are Polish groups.

Since atomless Borel probability measures are isomorphic to each other, the groups ${\rm Aut}(\gamma)$ are isomorphic 
as topological groups as $\gamma$ varies over atomless Borel probability measures. Similarly, the groups $L^0(\nu, {\mathbb T})$ 
are isomorphic to each other if $\nu$ is a finite non-zero atomless Borel measure. 
Notationally, our conventions will be as follows. We fix an atomless Borel probability measure $\gamma$ on $X$, and we will 
consider ${\rm Aut}(\gamma)$ only for this fixed $\gamma$. Since we will allow $\nu$ in $L^0(\nu, {\mathbb T})$ to 
have atoms in some situations, we reserve the letter $\lambda$ for an atomless Borel probability measure in 
$L^0(\lambda, {\mathbb T})$. In fact, as it is convenient to have a certain combinatorial structure on the space underlying 
$\lambda$, we will assume that $\lambda$ is the ``Lebesgue" measure on the Cantor set, that is, it is 
the product measure on $2^{\mathbb N} = \{ 0,1\}^{\mathbb N}$ of the measures on $\{ 0,1\}$ assigning equal weight of $1/2$ to 
each of the two points in $\{0, 1\}$. Taking $\lambda$ to be the Lebesgue measure on $[0,1]$ would be 
an equivalent acceptable choice. However, this choice would be less suitable for the combinatorics of our arguments.

Underlying spaces of Borel measures are standard Borel spaces. 
Hilbert spaces are taken over the scalar field of complex numbers $\mathbb C$; in particular, $L^2(\mu)$, 
for a finite Borel measure $\mu$, consists of complex valued, square integrable (measure equivalence classes of) functions 
on the space underlying $\mu$. By ${\mathcal U}(H)$ we denote the group of all unitary operators on the Hilbert space $H$ taken 
with the strong operator topology.

By convention, ${\mathbb N} = \{ n\in {\mathbb Z}\mid n>0\}$, so $0\not\in {\mathbb N}$ in this paper. We identify $n\in {\mathbb N}$
with the set $\{ 0, \dots, n-1\}$, that is, $n=\{ 0, \dots, n-1\}$; so, for example, $2=\{0 ,1\}$. 

Certain subgroups of $L^0(\lambda, {\mathbb T})$ will be used throughout, so it will be convenient to define them here. 
For a finite binary sequence $s\in 2^n = \{ 0,1\}^n$, for some $n\in {\mathbb N}$, let 
\begin{equation}\label{E:baca}
[s] = \{ \alpha\in 2^{\mathbb N}\mid \alpha\res n = s\}. 
\end{equation} 
Given $n$, we write 
\[
{\mathbb S}_n
\]
for the subgroup of $L^0(\lambda, {\mathbb T})$ consisting of all functions 
constant on each of the sets $[s]$ for $s\in 2^n$, that is, each element of ${\mathbb S}_n$ is of the form 
\[
\sum_{s\in 2^n} z_s\chi_{[s]} 
\]
with $z_s\in {\mathbb T}$, for $s\in 2^n$, and where $\chi_{[s]}$ is the indicator function of $[s]$. 
As a topological group ${\mathbb S}_n$ is 
isomorphic to ${\mathbb T}^{2^n}$, so it has the unique probability Haar measure, which we denote by $\theta$ independently of 
$n$, as $n$ will be always clear from the context. Elements of ${\mathbb S}_n$ will be denoted by $t$.

\subsection{The main theorem}\label{Su:mai} 

Let $Z$ be a {\bf Polish space}, that is, a completely metrizable separable space. 
We adopt 
the following linguistic convention. A property $\mathcal P$ of elements of $Z$ is said to hold for a {\bf generic} 
$z\in Z$ if there is a comeager set of $z\in Z$ that have property $\mathcal P$. 

We study closed subgroups of ${\rm Aut}(\gamma)$ generated by 
elements of ${\rm Aut}(\gamma)$, that is, subgroups of the form 
\[
\langle T\rangle_c= {\rm closure}\big(\{ T^n\mid n\in {\mathbb Z}\}\big),\; T\in {\rm Aut}(\gamma).
\]
One of the aims of the present paper is to answer the following question 
due to Glasner and Weiss (and reiterated in \cite{Ete} and
\cite[Question~1.3]{MT}) 
that has been circulating for a number of years. 

\noindent {\em Is it the case that, for a generic transformation $T\in {\rm Aut}(\gamma)$, the closed subgroup $\langle T\rangle_c$ 
generated by $T$ is isomorphic, as a topological group, to $L^0(\lambda, {\mathbb T})$?}

The answer to the question is negative. It follows from the more general theorem below.

\begin{theorem}\label{T:main}
For a generic transformation $T\in {\rm Aut}(\gamma)$, the closed group  
generated by $T$ is not included in 
the image of a continuous homomorphism from a topological group of the form $L^0(\nu, {\mathbb T})$, for a finite Borel measure $\nu$, to ${\rm Aut}(\gamma)$. 
\end{theorem} 

Theorem~\ref{T:main} is deduced from Theorem~\ref{T:notL}, a result on Koopman representations of 
$L^0(\lambda, {\mathbb T})$. Theorem~\ref{T:notL} will be stated in Section~\ref{S:thou} after we introduce the necessary notions in 
Sections~\ref{S:nole} and \ref{S:knre}.

To describe some context for Theorem~\ref{T:main}, we note that the behavior of a generic transformation $T\in {\rm Aut}(\gamma)$ 
is highly nonuniform. One only needs to recall the classical theorem of Rokhlin that conjugacy equivalence classes in 
${\rm Aut}(\gamma)$ are meager, see \cite[Theorem~2.5]{Kec2}, or its powerful strengthening---the theorem 
of Foreman and Weiss \cite[Corollary~13]{FW} on non-classifiability of the equivalence relation of conjugacy among generic 
$T\in {\rm Aut}(\gamma)$. In contrast to these results, a very different picture of a uniform behavior of the groups 
$\langle T\rangle_c$, for a generic $T$, had emerged including substantial evidence that
pointed to 
these groups being isomorphic to $L^0(\lambda, {\mathbb T})$. Results that were part of this picture had to do with the topological group 
structure of $\langle T\rangle_c$ and with the dynamics of 
$\langle T\rangle_c$:

\smallskip
\noindent \cite[Theorem~1]{Ag1}, \cite[Theorems~1 and 2]{Ag2}, \cite[Theorem~1.3]{Gla}, 
\cite[Theorems~3.11 and 5.2]{GW}, 
\cite[Theorem~1]{Ki2},  \cite[Corollary~3.8]{Mel}, 
\cite[Theorem~1.4]{MT}, \cite[Th{\'e}or{\`e}me~1.2.]{RSL}, 
\cite[Theorem~1, Corollary~2]{Sol2}, \cite[Theorem~1.3]{SE}, and \cite[Theorem~1.2]{Tih}. 
\smallskip

\noindent Additionally, certain groups analogous to $\langle T\rangle_c$, for a generic $T\in {\rm Aut}(\gamma)$, 
were determined to be isomorphic to $L^0(\lambda, {\mathbb T})$: \cite[Proposition~7]{LPT} and \cite[Theorems~1.2]{MT}.

The theorems mentioned above
may have been regarded as strong indications that 
a positive answer to the Glasner--Weiss question was to be expected. 
There was, however, another class of results on generic transformations $T$ that 
consisted of theorems concerned with the spectral behavior of  
such $T$: 

\smallskip
\noindent \cite[Theorem~6]{CN}, \cite[Theorem~1]{DJL}, \cite[Theorem 2.1 and Propositions 3.8 and 3.10]{Kat}, 
and \cite[Theorems~1 and 2]{Ste}. 
\smallskip

\noindent 
Even though these results did not involve groups $\langle T\rangle_c$ directly, 
it occurred to the author quite 
some time ago that they did not seem to point in the same direction as the above mentioned structural and dynamical theorems. 
This intuition turned out to be correct as, ultimately, it is the spectral results 
with which we reach a contradiction assuming that the answer to the Glasner--Weiss question is positive.

\subsection{A brief outline} 

The proof of Theorem~\ref{T:main} is based on the analysis of unitary representations of $L^0(\lambda, {\mathbb T})$ from 
\cite{Sol}; the main theorem of that paper is restated below as Theorem~\ref{T:rep}. 
The new theorem concerning unitary 
representations of $L^0(\lambda, {\mathbb T})$ proved here, 
Theorem~\ref{T:notL}, shows that the Koopman representations associated with ergodic boolean actions of $L^0(\lambda, {\mathbb T})$
fulfill an additional non-trivial condition. 
The proof of our main result, Theorem~\ref{T:main}, goes then by contradiction. 
Assuming that its conclusion fails, we show in Lemma~\ref{L:orto} 
that a certain type of boolean action of $L^0(\lambda, {\mathbb T})$ would have to exist. 
Then, Theorem~\ref{T:notL} is used to prove that a boolean action of this type does not exist. 
The ergodic theorem, Theorem~\ref{T:ert}, is used in the proof of Theorem~\ref{T:notL}.

\section{Known results on genericity in ${\rm Aut}(\gamma)$}\label{S:defre}

In this section, we recall some results on generic transformations in ${\rm Aut}(\gamma)$ that are relevant to our arguments 
later in the paper.

Theorem~\ref{T:spge} below describes the spectral property of generic transformations in ${\rm Aut}(\gamma)$, which 
was already briefly mentioned in Section~\ref{Su:mai}. 
Theorem~\ref{T:spge} asserts independence of maximal spectral types among sequences of 
powers of a generic transformation $T\in {\rm Aut}(\gamma)$. Its versions 
were proved in 
\cite[Theorem~6]{CN}, \cite[Theorem~1]{DJL}, \cite[Theorem 2.1 and Propositions 3.8 and 3.10]{Kat}, 
and \cite[Theorems~1 and 2]{Ste}. The statement below comes from the paper by del Junco and Lema{\'n}czyk \cite[Theorem~1]{DJL}.

\begin{theorem}[\cite{DJL}]\label{T:spge} 
Let $\nu(S)$ be the maximal spectral type of $S\in {\rm Aut}(\gamma)$.
For a generic transformation $T$ in ${\rm Aut}(\gamma)$ and 
$\ell_1, \dots , \ell_p, \ell'_1, \dots , \ell'_{p'}\in {\mathbb N}$, if the sequences 
$(\ell_1, \dots , \ell_p)$ and $(\ell'_1, \dots , \ell'_{p'})$ are not rearrangements of each other, then 
\[
\nu(T^{\ell_1}) \ast \cdots \ast \nu(T^{\ell_p}) \perp \nu(T^{\ell'_1}) \ast \cdots \ast \nu(T^{\ell'_{p'}}). 
\]
\end{theorem}

The next theorem states two properties of the group $\langle T\rangle_c$ for a generic $T\in {\rm Aut}(\gamma)$. 
A theorem implying point (i) was proved by Chacon and Schwartz\-bauer in \cite[Theorem~4.1]{CS}. 
Another proof of it can be found in 
\cite[Theorem~1.6]{MT}. Point (ii) follows from a result of Glasner and Weiss \cite[Theorem~5.2]{GW} and 
is explicitly proved in the paper by Melleray and Tsankov \cite[Theorem~1.4]{MT}. Recall 
that a topological group is {\bf extremely amenable} if all its continuous actions on compact spaces have fixed points.

\begin{theorem}\label{T:geco} 
The following two statements hold for a generic $T\in {\rm Aut}(\gamma)$. 
\begin{enumerate} 
\item[(i)]{\rm (\cite{CS})} $\langle T\rangle_c = \{ S\in {\rm Aut}(\gamma)\mid TS=ST\}$;

\item[(ii)] {\rm (\cite{GW}, \cite{MT})} $\langle T\rangle_c$ is extremely amenable. 
\end{enumerate}
\end{theorem}

The lemma below was proved in \cite[Lemma~3]{Sol2}. It relates global largeness of a set $B$ in ${\rm Aut}(\gamma)$ 
to its local largeness in $\langle T\rangle_c$, for a generic $T\in {\rm Aut}(\gamma)$, which makes it possible to 
turn properties of a generic $T$ into properties of $\langle T\rangle_c$ for a generic $T$. The lemma may be seen 
as an analogue of the Kuratowski--Ulam theorem. 

\begin{lemma}[\cite{Sol2}]\label{L:kuth}
Let $B\subseteq {\rm Aut}(\gamma)$ be a set with the Baire property. Then $B$ is comeager in ${\rm Aut}(\gamma)$ 
if and only if, for a generic $T\in {\rm Aut}(\gamma)$, the set $B\cap \langle T\rangle_c$ is comeager in $\langle T\rangle_c$. 
\end{lemma}

\section{Notions and lemmas on unitary representations of $L^0(\lambda, {\mathbb T})$}\label{S:nole} 

In this section, we build a framework, consisting of definitions and auxiliary lemmas, that is 
needed to carry out our arguments concerning unitary representations of $L^0(\lambda, {\mathbb T})$.
We proceed with care as some notions introduced here are new and some notions and 
notation from the literature are revised. For example, the definition of ${\mathbb N}[{\mathbb Z}^\times]$, the semigroup 
operation and the action of ${\mathbb Z}^\times$ on it are new.
The same goes for the notion of 
good permutation and good homeomorphism. Aside from being the basis of the arguments in this paper, 
the new framework 
allows us to streamline the statements of Theorems~\ref{T:rep} and \ref{T:dlco} proved in \cite{Sol} and \cite{Ete}, respectively. 
For example, the algebraic notions related to ${\mathbb N}[{\mathbb Z}^\times]$ permit a more concise formulation of 
Theorem~\ref{T:dlco}, cf.\;\cite[Theorem~4.4]{Ete}, while good homeomorphisms make it possible to remove an ad hoc linear order on 
$2^{\mathbb N}$ in the statement of Theorem~\ref{T:rep}, cf.\;\cite[Condition (A3) in Theorem~2.1]{Sol}.  

We start by defining equivariant Hilbert space maps, Section~\ref{Su:equi}. 
Then, we define a semigroup ${\mathbb N}[{\mathbb Z}^\times]$, Section~\ref{Su:semig}. 
Next, we describe 
a family of compact zero dimensional spaces $C_x$ indexed by $x\in {\mathbb N}[{\mathbb Z}^\times]$, Section~\ref{Suu:cx}. 
Each such space
comes with a natural notion of marginally compatible and compatible measures, Section~\ref{Su:compm}, 
and with a class of functions indexed by elements of $L^0$, Section~\ref{Su:cefu}. 
The measures and functions are then combined to define basic representations of $L^0$ that will 
be used later to build representations of interest, Section~\ref{Su:basic}.

\subsection{Equivariant Hilbert space maps}\label{Su:equi}
Throughout the paper, a central role will be played  by the following notions. 
Let $H_1$ and $H_2$ be Hilbert space with inner products 
$\langle \cdot, \cdot\rangle_1$ and $\langle \cdot, \cdot\rangle_2$, respectively. A function 
$p\colon H_1\to H_2$ is called a {\bf Hilbert space map} if it is linear and, for all $f, g\in H_1$, we have 
\[
\big\langle p(f), p(g)\big\rangle_2  = \langle f, g\rangle_1.
\]
Obviously, by the polarization identity, Hilbert space maps are simply linear isometries, in particular, 
a Hilbert space map $p\colon H_1\to H_2$ is a linear embedding of $H_1$ to $H_2$. 
Let $\xi_1, \, \xi_2$ be unitary representations of $L^0(\lambda, {\mathbb T})$ on 
$H_1$ and $H_2$, respectively. 
We say that $p$ is {\bf equivariant between $\xi_1$ and $\xi_2$} if, for each $\phi\in L^0(\lambda, {\mathbb T})$ and 
$f\in H_1$, we have 
\begin{equation}
p\big(\xi_1(\phi)\big( f\big)\big) = \xi_2(\phi)\big( p(f)\big). 
\end{equation}

\subsection{The semigroup ${\mathbb N}[{\mathbb Z}^\times]$ and the action of ${\mathbb Z}^\times$ on it}\label{Su:semig}

We view 
\[
{\mathbb Z}^\times = {\mathbb Z}\setminus \{ 0\}
\]
taken with multiplication as a semigroup. We will also need 
\[
{\mathbb Z}_2 = \{ -1, 1\}
\]
that is a subsemigroup of ${\mathbb Z}^\times$. 
Let 
\[
{\mathbb N}[{\mathbb Z}^\times] = \{ x\mid x \hbox{ a function},\, {\rm dom}(x)\not=\emptyset \hbox{ finite},\, 
{\rm dom}(x)\subseteq {\mathbb Z}^\times,\, \hbox{and } 
{\rm rng}(x) \subseteq {\mathbb N}\}. 
\]
We equip ${\mathbb N}[{\mathbb Z}^\times]$ with a binary operation $\oplus$ as follows. For $x, y \in {\mathbb N}[{\mathbb Z}^\times]$, let 
$x\oplus y$ 
be the element $z$ of ${\mathbb N}[{\mathbb Z}^\times]$ such that 
\[
{\rm dom}(z) = {\rm dom}(x)\cup {\rm dom}(y)
\]
and, for $k\in {\rm dom}(z)$, we let 
\[
z(k) = 
\begin{cases}
x(k)+y(k), &\text{ if }k\in {\rm dom}(x) \cap {\rm dom}(y);\\
x(k), &\text{ if } k\in {\rm dom}(x)\setminus {\rm dom}(y);\\
y(k), &\text{ if } k\in {\rm dom}(y)\setminus {\rm dom}(x).
\end{cases}
\]

We note that ${\mathbb Z}^\times$ and ${\mathbb N}[{\mathbb Z}^\times]$ are semigroups. 
There is a useful action of ${\mathbb Z}^\times$ on ${\mathbb N}[{\mathbb Z}^\times]$. For $x\in {\mathbb N}[{\mathbb Z}^\times]$ and 
$\ell \in {\mathbb Z}^\times$, let $\ell x$ be the element $z$ of ${\mathbb N}[{\mathbb Z}^\times]$ such that 
${\rm dom}(z) = \{ \ell m\mid m\in {\rm dom}(x)\}$ and,  for $k\in {\rm dom}(z)$, 
\[
z(k) = x(k/\ell). 
\]
Observe that, for $\ell, \ell_1, \ell_2\in {\mathbb Z}^\times$ and $x,y\in {\mathbb N}[{\mathbb Z}^\times]$, we have 
\[
\ell (x\oplus y) = \ell x\oplus \ell y\;\hbox{ and }\; \ell_2(\ell_1 x) = (\ell_2\ell_1) x. 
\]

\subsection{The topological spaces $C_x$, for $x\in {\mathbb N}[{\mathbb Z}^\times]$}\label{Suu:cx}

For $x\in {\mathbb N}[{\mathbb Z}^\times]$, we write 
\[
D(x) = \{ (k, i)\mid k\in {\rm dom}(x),\, 0\leq i<x(k)\}. 
\]
Let 
\[
C_x=(2^{\mathbb N})^{D(x)}.
\]

Define $\pi_{k, i}\colon C_x \to 2^{\mathbb N}$, for $(k, i)\in D(x)$, to be the projection 
from $C_x$ onto coordinate $(k,i)$. 
By a {\bf diagonal of $C_x$} we understand a set of the form 
\[
\{ \alpha\in C_x\mid \pi_{k,i}(\alpha) = \pi_{k',i'}(\alpha)\}, 
\]
for some distinct $(k,i), (k',i')\in D(x)$. 
We write 
\[
C^0_x
\]
for the set obtained from $C_x$ by removing the diagonals. 
For $n\in {\mathbb N}$, by an {\bf $n$-basic set for $x$} we understand a set of the form 
\begin{equation}\label{E:bafo}
\llbracket u\rrbracket = \{ \alpha\in C_x\mid \pi_{k,i}(\alpha)\res n = u(k,i) \hbox{ for all }(k,i)\in D(x) \}, 
\end{equation} 
where $u\colon D(x) \to 2^n$ is an injection. One can relate the set above to the sets defined in \eqref{E:baca} as follows 
\begin{equation}\label{E:bafalt} 
\llbracket u\rrbracket = \prod_{(k,i)\in D(x)} [u(k,i)].
\end{equation}
We leave proving the following easy lemma to the reader. 

\begin{lemma}\label{L:basi} 
For each $n_0\in {\mathbb N}$, the family of all $n$-basic sets for $x$ with $n\geq n_0$ forms a topological basis for $C^0_x$. 
\end{lemma} 

\noindent We call $n$-basic sets simply {\bf basic for $x$} if $n$ is not relevant or clear from the context.

For $x\in {\mathbb N}[{\mathbb Z}^\times]$, 
a permutation $\delta$ of $D(x)$ is called {\bf good} if, for each $(k,i)\in D(x)$, 
$\delta(k,i)= (k,j)$ for some $j$. Note that each good permutation $\delta$ induces a homeomorphism $\widetilde \delta$ of 
\[
C_x= (2^{\mathbb N})^{D(x)}
\]
by permuting coordinates, that is, for $\alpha\in C_x$, ${\widetilde \delta}(\alpha)\in C_x$ is determined by the following formulas 
\begin{equation}\label{E:dego}
\pi_{\delta(k,i)} \big( {\widetilde \delta}(\alpha)\big) = \pi_{k,i}(\alpha), \hbox{ for all }(k,i)\in D(x).
\end{equation} 
We call such homeomorphisms $\widetilde \delta$ {\bf good homeomorphisms of $C_x$}. Observe that both good permutations of $D(x)$ 
and good homeomorphisms of $C_x$ form groups under composition. 

We make the following observation on the connection between basic sets and good homeomorphisms that 
will be used in the proof of Lemma~\ref{L:ident}. 

\begin{lemma}\label{L:basper} 
Let $U$ be a set that is $n$-basic for $x$. Then 
the sets ${\widetilde\delta}( U)$ are pairwise disjoint $n$-basic sets when $\delta$ varies 
over all good permutations of $D(x)$. 
\end{lemma} 

\begin{proof} Let $U=\llbracket u\rrbracket$ for an injection $u\colon D(x)\to 2^n$. 
If $\delta$ is a good permutation of $D(x)$, then, by \eqref{E:dego}, we have 
\[
{\widetilde\delta}\big( \llbracket u\rrbracket\big) = \llbracket u\circ \delta^{-1}\rrbracket,
\]
and the conclusion follows. 
\end{proof}

The space $C_{x\oplus y}$ can be naturally seen as $C_x\times C_y$, in fact, in several ways. To make these identifications precise, fix 
\begin{equation}\label{E:seq} 
{\bar \iota}= (\iota^x, \iota^y), 
\end{equation}
where $\iota^x\colon D(x)\to D(x\oplus y)$ 
and $\iota^y\colon D(y)\to D(x\oplus y)$ are injections with disjoint images and such that, for each $(k,i) \in D(x)$, 
$\iota^x(k,i)= (k,j)$ for some $j$, and, similarly, for each $(k,i)\in D(y)$, 
$\iota^y(k,i)= (k,j)$ for some $j$. Note that $D(x\oplus y)$ is a disjoint union of $\iota^x\big(D(x)\big)$ and $\iota^y\big(D(y)\big)$. 
Then, define 
\begin{equation}\label{E:hom}
h_{\bar\iota}\colon C_x\times C_y \to  C_{x\oplus y}, \; h_{\bar\iota}(\alpha, \beta)=\gamma, 
\end{equation}
where $\gamma$ is given as follows. To specify $\gamma\in C_{x\oplus y}$, it suffices to specify $\pi_{k,j}(\gamma)$ for 
$(k,j)\in D(x\oplus y)$. For each such $(k,j)$, there exists a unique $(k,i)$ with 
\[
(k,j) = \iota^x(k,i)\;\hbox{ or }\; (k,j) = \iota^y(k,i)
\]
and not both. In the first case, we let 
\[
\pi_{k,j}(\gamma) = \pi_{k,i}(\alpha), 
\]
while in the second
\[
\pi_{k,j}(\gamma) = \pi_{k,i}(\beta). 
\]

We note that $h_{\bar\iota}$ is a homeomorphism. The following lemma concerns interactions of $h_{\bar\iota}$ with basic sets. We leave its verification to the reader. 

\begin{lemma}\label{L:intp} 
Let $\bar\iota$ be as in \eqref{E:seq}. If $u\colon D(x\oplus y)\to 2^n$ is an injection, 
then
\[
h_{\bar\iota}\big( \llbracket u\circ \iota^x\rrbracket \times  \llbracket u\circ\iota^y\rrbracket\big)  = \llbracket u\rrbracket.
\]
\end{lemma}

For $\ell \in {\mathbb Z}^\times$, we define 
\begin{equation}\label{E:hom2}
e_{x,\ell}\colon C_x\to  C_{\ell x}, \; e_{x,\ell}(\alpha)=\gamma, 
\end{equation}
where for $(k,i) \in  D(\ell x)$, we let $\gamma(k, i) = \alpha(k/\ell, i)$. We note that $e_{x,\ell}$ is a homeomorphism.

\subsection{Marginally compatible and compatible measures on $C_x$}\label{Su:compm} 

Let $\mu$ be a finite Borel measure on $C_x$. We say that $\mu$ is 
 {\bf marginally compatible with} $x\in {\mathbb N}[{\mathbb Z}^\times]$ if 
the marginal measures $(\pi_{k,i})_*(\mu)$ of $\mu$ on $2^{\mathbb N}$, for $(k,i) \in D(x)$,
are absolutely continuous with respect to $\lambda$. 
We say that $\mu$ is {\bf compatible with} $x\in {\mathbb N}[{\mathbb Z}^\times]$ if 
\begin{enumerate}
\item[(a)] $\mu$ is marginally compatible with $x$; 

\item[(b)]  $\mu$ is invariant under good homeomorphisms of $C_x$; 

\item[(c)] all diagonals of $C_x$ have measure zero with respect to $\mu$, that is, $\mu$ concentrates on $C^0_x$. 
\end{enumerate}
The condition of marginal compatibility, that is, (a), is needed for representations as in \eqref{E:pire} to be well defined, 
while conditions (b) and (c) ensure uniqueness in Theorem~\ref{T:rep}. 

For each $x \in {\mathbb N}[{\mathbb Z}^\times]$, let ${\mathcal M}_x$ be the set of measures on $C_x$ compatible with
$x$, and let 
\[
{\mathcal M} = \bigcup_{x\in {\mathbb N}[{\mathbb Z}^\times]} {\mathcal M}_x.
\]

For $\mu, \nu\in {\mathcal M}$, with $\mu$ compatible with $x$ and $\nu$ compatible with $y$, we define 
\[
\mu\otimes \nu = 
\sum_{\bar\iota} (h_{\bar\iota})_*(\mu\times \nu), 
\]
where $h_{\bar\iota}$ are the homeomorphisms defined in \eqref{E:hom} and ${\bar\iota}$ ranges over pairs as in \eqref{E:seq}. 
\begin{lemma}\label{L:compa} 
Let $\mu$ and $\nu$ be measures compatible with $x$ and $y$, respectively. 
\begin{enumerate}
\item[(i)] For each ${\bar\iota}$ as in \eqref{E:seq}, the measure $(h_{\bar\iota})_*(\mu\times \nu)$ fulfills (a) and (c) 
from the definition of compatibility with $x\oplus y$. 

\item[(ii)] The measure $\mu\otimes \nu$ is compatible with $x\oplus y$. 
\end{enumerate}
\end{lemma}

\begin{proof} Point (ii) follows immediately from (i). In (i), checking (a) is straightforward since, for $(k,j)\in D(x\oplus y)$, 
\[
(\pi_{k,j})_*\big( (h_{\bar\iota})_*(\mu\times \nu)\big)  = (\pi_{k,i})_*(\mu)\;\hbox{ or }\;
(\pi_{k,j})_*\big( (h_{\bar\iota})_*(\mu\times \nu)\big)  = (\pi_{k,i})_*(\nu),
\]
for an appropriate $(k,i)\in D(x)$ or $(k,i)\in D(y)$. 

It remains to see (c). 
Let $(k,j), (k',j')\in D(x\oplus y)$ be distinct, and consider the diagonal 
\[
\Delta = \{ \gamma\in C_{x\oplus y} \colon \pi_{k,j}(\gamma)= \pi_{k',j'}(\gamma)\}.
\]
We need to check that 
\begin{equation}\label{E:chdi}
(\mu\times \nu)\big(h_{\bar \iota}^{-1}(\Delta)\big) =0.
\end{equation} 
Let ${\bar \iota} = (\iota^x, \iota^y)$. Because of symmetry, we only need to consider two cases 
\[
\iota^x(k,i) = (k,j),\; \iota^x(k',i')= (k',j'), \hbox{ for some } (k,i), (k',i') \in D(x),
\]
and 
\[
\iota^x(k,i) = (k,j),\; \iota^y(k',i')= (k',j'), \hbox{ for some } (k,i)\in D(x)\hbox{ and } (k',i')\in D(y). 
\]
Note that in either case $(k, i)$ and $(k',i')$ are distinct from each other since $(k,j)$ and $(k',j')$ are distinct. 
Now, in the first case, we have 
\[
h_{\bar \iota}^{-1}(\Delta) = \{ \alpha\in C_x\mid \pi_{k,i}(\alpha) = \pi_{k',i'}(\alpha)\} \times C_y, 
\]
and \eqref{E:chdi} follows from $\mu$ being compatible with $x$, as this property implies 
\[
\mu\big(\{  \alpha\in C_x\mid \pi_{k,i}(\alpha) = \pi_{k',i'}(\alpha)\}\big)=0. 
\]
In the second case, we have 
\[
h_{\bar \iota}^{-1}(\Delta) = \{ (\alpha, \beta) \in C_x\times C_y\mid \pi_{k,i}(\alpha) = \pi_{k',i'}(\beta)\}. 
\]
Fubini's theorem implies that to prove \eqref{E:chdi} for the set above, it suffices to see that, for each $\beta_0\in 2^{\mathbb N}$, 
\[
\mu \big(\{ \alpha\in C_x\mid \pi_{k,i}(\alpha) = \beta_0\}\big) =0.
\]
But otherwise, we would have 
\[
\big((\pi_{k,i})_*(\mu)\big)\big(\{ \beta_0\}\big)>0, 
\]
contradicting absolute continuity of the marginal measure $(\pi_{k,i})_*(\mu)$ with respect to $\lambda$ in light of 
$\lambda$ not having atoms. 
\end{proof}

By inspecting the definitions, we see the following lemma. 

\begin{lemma}\label{L:abseio} 
\begin{enumerate}
\item[(i)] ${\mathcal M}$ with the operation $\otimes$ is an abelian semigroup. 

\item[(ii)] Let $\mu, \mu'\in {\mathcal M}$ be compatible with $x$, and let $\nu, \nu'\in {\mathcal M}$ be compatible with $y$. 
If $\mu \preceq \mu'$ and $\nu\preceq \nu'$, then $\mu\otimes \nu\preceq \mu'\otimes \nu'$. 
\end{enumerate} 
\end{lemma} 

Again ${\mathbb Z}^\times$ acts on ${\mathcal M}$ as follows. 
For $\mu\in {\mathcal M}$ compatible with $x$, let 
\[
\ell \mu = (e_{x,\ell})_*(\mu),
\]
where $e_{x,\ell}$ is the homeomorphism given by \eqref{E:hom2}. The following lemma is easy to check. 
\begin{lemma}\label{L:compa2} 
The measure $\ell\mu$ is compatible with $\ell x$. 
\end{lemma} 

The next lemma is also checked by a quick inspection. 
\begin{lemma}\label{L:abse} 
\begin{enumerate}
\item[(i)] For $\mu, \nu\in {\mathcal M}$ and $\ell, \ell_1, \ell_2\in {\mathbb Z}^\times$,  
we have $\ell (\mu\otimes \nu) = \ell \mu\otimes \ell \nu$ and $\ell_2(\ell_1 \mu) = (\ell_2\ell_1) \mu$. 

\item[(ii)] Let $\mu, \mu'\in {\mathcal M}$ be compatible with $x$ and let $\ell\in {\mathbb Z}^\times$. 
If $\mu \preceq \mu'$, then $\ell \mu \preceq \ell \mu'$. 
\end{enumerate} 
\end{lemma}

\subsection{Certain functions on $C_x$}\label{Su:cefu} 

Let $x\in {\mathbb N}[{\mathbb Z}^\times]$. For $\phi\in L^0(\lambda, {\mathbb T})$, define a function $R_x(\phi)$ on $C_x$ 
with values in $\mathbb T$ by the formula 
\begin{equation}\label{E:gogo}
R_x(\phi)= \prod_{(k,i)\in D(x)} (\phi\circ \pi_{k, i})^k. 
\end{equation}
Note that since $\phi$ is a measure class of functions with respect to $\lambda$, $R_x(\phi)$ is defined  
only up to a set of the form 
\begin{equation}\notag
\bigcup_{(k,i)\in D(x)} \pi_{k,i}^{-1}(A_{k,i}), 
\end{equation} 
where $A_{k,i}$ is a subset of $2^{\mathbb N}$ with $\lambda( A_{k,i})=0$.  
In particular, for each measure $\mu$ marginally compatible with $x$, $R_x(\phi)$ determines a measure class of functions 
with respect to $\mu$. 
Functions $R_x(\phi)$ will be crucial in defining unitary representations of $L^0(\lambda, {\mathbb T})$ in 
Hilbert spaces $L^2(\mu)$ for such $\mu$. 

We prove now three lemmas establishing some properties of the functions $R_x(\phi)$. 
The first of these lemmas will be important in computations. 

\begin{lemma}\label{L:pfco} 
Let $u\colon D(x)\to 2^n$ be an injection, for $x\in {\mathbb N}[{\mathbb Z}^\times]$ and $n\in {\mathbb N}$. 
Fix $z_s\in {\mathbb T}$, for $s\in 2^n$, and consider 
\[
\phi = \sum_{s\in 2^n} z_s\chi_{[s]} \in {\mathbb S}_n,
\]
where $\chi_{[s]}$ is the indicator function of the set $[s]\subseteq 2^{\mathbb N}$. 
For this $\phi$, the function $R_x(\phi)$ 
is constant on $\llbracket u\rrbracket$, and its constant value is 
\[
\prod_{(k,i)\in D(x)} z_{u(k,i)}^k. 
\]
\end{lemma}

\begin{proof} 
By \eqref{E:bafo}, for $\alpha\in \llbracket u\rrbracket$, we have 
\[
\phi(\pi_{k, i}(\alpha)) = z_{u(k,i)},
\]
and the conclusion follows.
\end{proof} 

The next two lemmas describe interactions of the functions in \eqref{E:gogo} with homeomorphisms defined before. 

\begin{lemma}\label{L:sym} 
For each $\phi\in L^0(\lambda, {\mathbb T})$, $R_x(\phi)$ is invariant under all good homeomorphisms of $C_x$. 
\end{lemma}

\begin{proof} 
Fix a good permutation $\delta$ of $D(x)$ and $\alpha\in C_x$. We show that 
\[
\prod_{(k,i)\in D(x)} (\phi\circ \pi_{k, i})^k\big({\widetilde \delta}(\alpha)\big) = \prod_{(k,i)\in D(x)} (\phi\circ \pi_{k, i})^k\big(\alpha\big). 
\]
It suffices to prove that, given $k_0$ such that $(k_0, i)\in D(x)$ for some $i$, we have 
\begin{equation}\label{E:kuar}
\prod_{(k_0,i)\in D(x)} (\phi\circ \pi_{k_0, i})^{k_0}\big({\widetilde \delta}(\alpha)\big) = 
\prod_{(k_0,i)\in D(x)} (\phi\circ \pi_{{k_0}, i})^{k_0}\big(\alpha\big).
\end{equation} 
Since the right-hand side of \eqref{E:kuar} is equal to 
\[
\big( \prod_{(k_0,i)\in D(x)} \phi\circ \pi_{k_0, i}(\alpha)\big)^{k_0}
\]
while, by \eqref{E:dego}, the left-hand side of \eqref{E:kuar} is equal to 
\[
\Big( \prod_{(k_0,i)\in D(x)} \phi\circ \pi_{k_0, i}\big({\widetilde \delta}(\alpha)\big)\Big)^{k_0}\\
= \big( \prod_{(k_0,i)\in D(x)} \phi\circ \pi_{\delta^{-1}(k_0, i)}(\alpha)\big)^{k_0}, 
\]
it suffices to notice that 
\[
 \prod_{(k_0,i)\in D(x)} \phi\circ \pi_{\delta^{-1}(k_0, i)}(\alpha) =  \prod_{(k_0,i)\in D(x)} \phi\circ \pi_{k_0, i}(\alpha),
\]
which is clear as $\delta^{-1}$ is a good permutation of $D(x)$. 
\end{proof}

\begin{lemma}\label{L:idio}
Let $x,y\in {\mathbb N}[{\mathbb Z}^\times]$ and let $\phi\in L^0(\lambda, {\mathbb T})$. 
\begin{enumerate} 
\item[(i)] For ${\bar\iota}$ as in \eqref{E:seq} and $\alpha\in C_x$, $\beta\in C_y$, 
\[
\big( R_{x\oplus y} (\phi) \circ h_{\bar\iota}\big) (\alpha, \beta) = R_x(\phi)(\alpha)\, R_y(\phi)(\beta).
\]

\item[(ii)] For $\ell\in {\mathbb Z}^\times$, 
\[
R_{\ell x}(\phi) \circ e_{x,\ell} = R_x(\phi)^\ell. 
\]
\end{enumerate} 
\end{lemma} 

\begin{proof} Point (ii) is clear. As for point (i), let ${\bar\iota}= (\iota^x, \iota^y)$. 
Using the definitional identities for $h_{\bar\iota}$, 
\[
\pi_{\iota^x(k,i)}(h_{\bar\iota}(\alpha, \beta)) = \pi_{k,i}(\alpha)\;\hbox{ and }\; \pi_{\iota^y(k,i)}(h_{\bar\iota}(\alpha, \beta))= \pi_{k,i}(\beta),
\] 
for $(k,i)\in D(x)$ and $(k,i)\in D(y)$, respectively, 
we get 
\[
\begin{split}
&\big( \prod_{(k,i)\in D(x\oplus y)} (\phi\circ \pi_{k, i})^k\big) (h_{\bar\iota}(\alpha, \beta))\\
&= \big( \prod_{(k,i)\in D(x)} (\phi\circ \pi_{\iota^x(k, i)})^k\big) (h_{\bar\iota}(\alpha, \beta))\, \big( \prod_{(k,i)\in D(y)} (\phi\circ \pi_{\iota^y(k, i)})^k\big) (h_{\bar\iota}(\alpha, \beta))\\
&=\big( \prod_{(k,i)\in D(x)} (\phi\circ \pi_{k, i})^k\big) (\alpha) \,\big(\prod_{(k,i)\in D(y)} (\phi\circ \pi_{k, i})^k\big) (\beta),
\end{split} 
\]
as required. 
\end{proof}

\subsection{Basic representations of $L^0(\lambda, {\mathbb T})$}\label{Su:basic} 

We define here certain unitary representations of $L^0(\lambda, {\mathbb T})$. 
All other representations of $L^0(\lambda, {\mathbb T})$ 
in this paper, except for Koopman representations, are built from the ones defined in this section. A good reason for this situation 
is given in Theorem~\ref{T:rep}. 

Let $x\in {\mathbb N}[{\mathbb Z}^\times]$. 

Let $\mu$ be a measure on $C_x$ that is marginally compatible with $x$. We define a unitary representation 
$\rho_x$ of $L^0(\lambda, {\mathbb T})$ on $L^2(\mu)$ by letting $\rho_x(\phi)$, for $\phi\in L^0(\lambda, {\mathbb T})$,  
be the multiplication operator on $L^2(\mu)$ given by 
\begin{equation}\label{E:pire}
\rho_x(\phi)\big( f\big) = R_x(\phi) f,\; \hbox{ for }f \in L^2(\mu),
\end{equation} 
where on the right-hand side is the product of a function given by \eqref{E:gogo} and $f$. 
The assumption on $\mu$ of being marginally compatible with $x$ ensures that the function on the right hand side is measurable 
with respect to $\mu$. It is now easy to check that $\rho_x(\phi)$ in \eqref{E:pire} is a unitary operator on $L^2(\mu)$ and that 
\begin{equation}\label{E:relt}
\rho_x\colon L^0(\lambda, {\mathbb T})\to {\mathcal U}\big( L^2(\mu)\big)
\end{equation}
is a unitary representation of $L^0(\mu, {\mathbb T})$. 

Note that strictly speaking $\rho_x$ depends also on $\mu$. We do not reflect this fact in our notation, that is, we use 
the same piece of notation $\rho_x$ to denote unitary representations on $L^2(\mu)$ for all measures $\mu$ 
marginally compatible with $x$ as it will not cause confusion.

Let $\mu$ be a measure compatible with $x$. We define now a subrepresentation of $\rho_x$ on $L^2(\mu)$. 
We let 
\begin{equation}\label{E:spsy}
\widetilde{L^2}(\mu)
\end{equation}
be the closed subspace of $L^2(\mu)$ consisting of all (equivalence classes of) 
functions invariant under good homeomorphisms of $C_x$. Lemma~\ref{L:sym} implies that, for each $\phi\in L^0(\lambda, {\mathbb T})$, 
the function $\rho_x(\phi)\big( f \big)$, as defined by \eqref{E:pire}, is an element of $\widetilde{L^2}(\mu)$ if $f$ is. This means that  
$\widetilde{L^2}(\mu)$ 
is a subspace of $L^2(\mu)$ invariant under the representation $\rho_x$ from \eqref{E:relt}. We denote the restriction of 
$\rho_x$ to $\widetilde{L^2}(\mu)$ by the same latter, that is, 
\begin{equation}\label{E:rhsy}
\rho_x\colon  L^0(\lambda, {\mathbb T}) \to {\mathcal U}\big(\widetilde{L^2}(\mu)\big).
\end{equation} 

Two instances of the representation $\rho_x$ that we will use most are the above instance on the space 
$\widetilde{L^2}(\mu)$, for a measure $\mu$ compatible with $x$, and the instance on the space 
$L^2(\mu\res U)$, where $\mu$ is a measure compatible with $x$ and $U$ is a basic set for $x$. Here, by $\mu\res U$, 
we understand a Borel measure on $C_x$ defined by 
\[
\big( \mu\res U\big) (A) = \mu(A\cap U), 
\]
for Borel subsets $A$ of $C_x$. Note that here $\mu\res U$ is marginally compatible with $x$. 
The next lemma describes an embedding between these two representations.

\begin{lemma}\label{L:ident} 
Let $\mu$ be a measure compatible with $x$, and let $U$ be a basic set for $x$. 
There exists a Hilbert space map 
\begin{equation}\label{E:mapd}
L^2(\mu\res U)\ni f\to {\widetilde f}\in  \widetilde{L^2}(\mu)
\end{equation} 
that is equivariant between $\rho_x$ in $L^2(\mu\res U)$ and $\rho_x$ in $\widetilde{L^2}(\mu)$. 
\end{lemma} 

\begin{proof} 
Fix $f\in L^2(\mu\res U)$. Now, $f$ is an equivalence class of functions on $C_x$ measurable with respect to $\mu\res U$. Pick 
a representative $F$ in the class $f$ such that 
\begin{equation}\label{E:repfu} 
F\res (C_x\setminus U) = 0. 
\end{equation}
Such an $F$ can always be found. Define 
\begin{equation}\label{E:averd}
{\widetilde f}= \frac{1}{N} \sum_\delta \big( F\circ \widetilde{\delta}^{-1}\big)\in \widetilde{L^2}(\mu),
\end{equation} 
where the sum is taken over all good permutations of $D(x)$ and $N$ is the number of good permutations of $D(x)$. 
It is clear that ${\widetilde f}\in \widetilde{L^2}(\mu)$; in fact, 
\eqref{E:averd} specifies the same element of $\widetilde{L^2}(\mu)$ regardless of which representative $F$ of $f$ 
is chosen as long as it fulfills \eqref{E:repfu}. Thus, the map \eqref{E:mapd} is well defined. 
The conclusion that \eqref{E:mapd} 
is a Hilbert space map follows from Lemma~\ref{L:basper} and the invariance of $\mu$ under good homeomorphisms of $C_x$. 
The equivariance of the map \eqref{E:mapd} is a consequence of Lemma~\ref{L:sym}. 
\end{proof}

\section{Known results on unitary representations of $L^0(\lambda, {\mathbb T})$}\label{S:knre} 

Theorem~\ref{T:rep} below gives a general form of a unitary representation of 
$L^0(\lambda, {\mathbb T})$. It was proved in \cite[Theorem~2.1]{Sol}. The framework of the statement below 
differs somewhat from that in \cite{Sol}; for example, good homeomorphisms introduced in this paper together 
with the definition of the space $\widetilde{L^2}(\mu)$ allow us 
to remove an arbitrarily chosen linear order on $2^{\mathbb N}$ from the formulation of the theorem. 
However, translating the statement from \cite{Sol} into the one below is not difficult.
One may view Theorem~\ref{T:rep} as a spectral theorem for unitary representations of $L^0(\lambda, {\mathbb T})$. 
This theorem forms the basis of our proof of Theorem~\ref{T:main}. 

Recall the definitions \eqref{E:pire}, \eqref{E:spsy},  and \eqref{E:rhsy}.

\begin{theorem}[\cite{Sol}]\label{T:rep} 
Let $\xi\colon L^0(\lambda, {\mathbb T})\to {\mathcal U}(H)$ be a unitary representation on a separable Hilbert space $H$. 
Let $H_0$ be the orthogonal complement of the subspace of $H$ consisting of vectors fixed by the representation. 
Then the representation restricted to $H_0$ 
is determined by a sequence of finite Borel measures 
$(\mu^j_x)_{x\in {\mathbb N}[{\mathbb Z}^\times], j\in {\mathbb N}}$ such that, for each $j$, 
\[
\mu_x^j\hbox{ is a measure on $C_x$ compatible with $x$, and }  \mu^{j+1}_x\preceq \mu^j_x. 
\]
The representation restricted to $H_0$ is isomorphic to the $\ell^2$-sum over 
$x\in {\mathbb N}[{\mathbb Z}^\times]$ and $j\in {\mathbb N}$ of the representations 
\begin{equation}\label{E:pff}
L^0(\lambda, {\mathbb T})\times \widetilde{L^2}(\mu^j_x) \ni (\phi, f) 
\to \rho_x(\phi)\big( f \big)\in \widetilde{L^2}(\mu^j_x). 
\end{equation}
Furthermore, the sequence $(\mu^j_x)_{x\in {\mathbb N}[{\mathbb Z}^\times], j\in {\mathbb N}}$ is unique up to mutual absolute continuity of its entries. 
\end{theorem}

\noindent {\bf From this point on, given a unitary representation $\xi\colon L^0(\lambda, {\mathbb T})\to {\mathcal U}(H)$, 
for a separable Hilbert space $H$, we write}
{\bf \[
\mu_x = \mu^1_x,\hbox{ for }x\in {\mathbb N}[{\mathbb Z}^\times]. 
\]}
We neglect to indicate the dependence of $\mu_x$ on $\xi$, as the representation will always be clear from the context. 

Theorem~\ref{T:dlco} below was proved by Etedadialiabadi in his PhD thesis and published in \cite[Theorem~4.4]{Ete}. The framework of point (ii) of this theorem as stated below, 
in particular, the algebraic 
notions associated with ${\mathbb N}[{\mathbb Z}^\times]$,  are new here and somewhat different from the framework in \cite{Ete}; 
passing between the two formulations is not difficult. 
Given a unitary representation $\xi$ of $L^0(\lambda, {\mathbb T})$, Theorem~\ref{T:dlco} 
translates the condition asserting that, 
for a generic $\phi\in L^0(\lambda, {\mathbb T})$, the maximal 
spectral type of the operator $\xi(\phi)$ satisfies the property in Theorem~\ref{T:spge} 
into a condition on the sequence of measures associated with $\xi$ by Theorem~\ref{T:rep}.

\begin{theorem}[\cite{Ete}]\label{T:dlco} 
Let $\xi\colon L^0(\lambda, {\mathbb T})\to {\mathcal U}(H)$ be a unitary representation on a separable Hilbert space $H$. 
For $\psi\in L^0(\lambda, {\mathbb T})$, let $\nu(\psi)$ be the maximal spectral type of the unitary operator $\xi(\psi)$. 
The following two conditions are equivalent. 
\begin{enumerate} 
\item[(i)] For a generic element $\phi$ of $L^0(\lambda, {\mathbb T})$ and 
$\ell_1, \dots , \ell_p, \ell'_1, \dots , \ell'_{p'}\in {\mathbb N}$, if the sequences 
$(\ell_1, \dots , \ell_p)$ and $(\ell'_1, \dots , \ell'_{p'})$ are not rearrangements of each other, then 
\[
\nu(\phi^{\ell_1}) \ast \cdots \ast \nu(\phi^{\ell_p}) \perp \nu(\phi^{\ell'_1}) \ast \cdots \ast \nu(\phi^{\ell'_{p'}}). 
\]

\item[(ii)] For 
$\ell_1, \dots , \ell_p, \ell'_1, \dots , \ell'_{p'}\in {\mathbb Z}^\times$ and $x_1, \dots , x_p, x'_1, \dots , x'_{p'}\in {\mathbb N}[{\mathbb Z}^\times]$, 
if 
\[
\ell_1 x_1 \oplus \cdots \oplus \ell_p x_p = \ell'_1x'_1\oplus \cdots \oplus \ell'_{p'} x'_{p'} 
\]
and the sequences $(\ell_1, \dots , \ell_p)$, $(\ell'_1, \dots , \ell'_{p'})$ are not rearrangements of each other, then 
\[
\ell_1 \mu_{x_1} \otimes \cdots \otimes \ell_p \mu_{x_p} \perp \ell'_1\mu_{x'_1}\otimes \cdots \otimes \ell'_{p'} \mu_{x'_{p'}}. 
\]
\end{enumerate} 
\end{theorem}

\section{The theorem on Koopman representations and an outline of its proof}\label{S:thou}

Theorem~\ref{T:notL} below gives an additional non-trivial condition fulfilled by unitary representations of 
$L^0(\lambda, {\mathbb T})$ that arise from 
boolean actions of this group as follows. 
To state it, we need to recall some notions related to boolean actions that will also 
be used later on in the paper.

Let $G$ be a Polish group. A {\bf boolean action of $G$ on $(X,\gamma)$} is a continuous homomorphism 
$\zeta \colon G\to {\rm Aut}(\gamma)$. The word action is justified by viewing $G$ as acting via $\zeta$ on the boolean 
algebra of measure classes of measurable subsets of $(X, \gamma)$ as follows 
\[
gB = \zeta(g)\big(B\big). 
\]
We point out that, by \cite[Proposition~1.3]{GW}, a boolean action of a 
Polish group is induced by a near-action, as defined below, 
and vice-versa each near-action induces a boolean action. 
Recall from \cite[Definition~1.2]{GW} that a {\bf near-action} of $G$ on $(X, \gamma)$ is a Borel map
$G \times X \to X$, $(g, \omega) \to  g\omega$, with the following properties: 
\begin{enumerate}
\item[---] $1\omega = \omega$ for almost every $\omega\in X$ with respect to $\gamma$; 

\item[---] for $g, h\in G$, $g(h\omega) = (gh)\omega$, for almost every $\omega\in X$ with respect to $\gamma$, 
where the set of points $\omega$ 
for which this equality holds depends on $g, h$;

\item[---] the map $X\ni \omega\to g\omega\in X$ is measure preserving with respect to $\gamma$, for each $g\in G$. 
\end{enumerate}
It follows from the above discussion that expressions of the form 
$g \omega$, for $g\in G$ and $\omega\in X$, make sense for a boolean action of $G$ as they are understood in terms of 
a near-action realizing the boolean action. Obviously, for two near-actions realizing the same boolean action and for $g\in G$, 
the values $g\omega$ coincide only on a set of $\omega\in X$ that has measure $1$ with respect to $\gamma$. 

A boolean action of $G$ on $(X,\gamma)$ is called {\bf ergodic} if, for each measure class $B$ of a measurable set in 
$(X, \gamma)$ such that $gB=B$ for each $g\in G$, we have $\gamma(B)=1$ or $\gamma(B)=0$.

Given a boolean action $\zeta\colon L^0(\lambda, {\mathbb T}) \to {\rm Aut}(\gamma)$, the 
{\bf Koopman representation associated with $\zeta$} is the unitary representation of $L^0(\lambda, {\mathbb T})$ 
on $L^2(\gamma)$ given by 
\[
L^0(\lambda, {\mathbb T})\ni \phi\to U_\phi\in {\mathcal U}\big(L^2(\gamma)\big), 
\]
where, for $f\in L^2(\gamma)$,  
\begin{equation}\notag
U_\phi(f)(\omega) =  f(\phi^{-1}\omega), \hbox{ for }\omega\in X. 
\end{equation} 
On the right-hand side of formula 
above, we use a near-action inducing $\zeta$.

\begin{theorem}\label{T:notL}
Assume that a unitary representation of $L^0(\lambda, {\mathbb T})$ is the Koopman representation associated with  
an ergodic boolean action of $L^0(\lambda, {\mathbb T})$. 
Then, for $\ell_1, \dots , \ell_p\in {\mathbb Z}_2$ and $x_1, \dots , x_p, x\in {\mathbb N}[{\mathbb Z}^\times]$
with
\[
\ell_1x_1 \oplus \cdots \oplus  \ell_p x_p = x,
\]
we have 
\[
\ell_1\mu_{x_1} \otimes \cdots \otimes \ell_p\mu_{x_p} \preceq \mu_x. 
\]
\end{theorem}

Note that in the theorem above the coefficients $\ell_1, \dots, \ell_p$ are restricted to come from ${\mathbb Z}_2$ rather than 
from the whole ${\mathbb Z}^\times$.

We outline the course of our argument proving Theorem~\ref{T:notL}. 
Let $\sigma$ be the Koopman representation induced by a boolean action of $L^0(\lambda, {\mathbb T})$ on $(X,\gamma)$.
Let $H_0(\gamma)$ be the orthogonal complements of the space consisting of the elements of 
$L^2(\gamma)$ that are fixed by $\sigma$. By ergodicity of the boolean action, 
$H_0(\gamma)$ consist of elements of $L^2(\gamma)$ with zero integral. 
Let $\tau$ be the representation gotten for $\sigma$ by applying Theorem~\ref{T:rep}.. We denote the Hilbert space 
underlying $\tau$ by $H_\tau$. Of 
course, $\tau$ comes in the form of finite measures $\mu^j_x$, for $j\in {\mathbb N}$ and $x\in {\mathbb N}[{\mathbb Z}^\times]$, 
and $\mu_x = \mu^1_x$ for all $x$. 
Finally, let $\Phi\colon H_0(\gamma)\to H_\tau$ be the isomorphism of representations produced in Theorem~\ref{T:rep}. 
Recall also the notion of equivariant Hilbert space map from Section~\ref{Su:equi}.

By the results of Section~\ref{Su:compm}, to prove Theorem~\ref{T:notL},  
it suffices to show only two special cases of the conclusion, that is, for all $x,y, z\in {\mathbb N}[{\mathbb Z}^\times]$, 
\[
\mu_y\otimes \mu_z\preceq \mu_x, \hbox{ if } y\oplus z= x,\; \hbox{ and }\; (-1)\mu_{x} \preceq \mu_{(-1)x}.
\]
Below, we focus on the first case only; the second case is handled by analogous methods. By the definition of 
$\mu_y\otimes \mu_z$ and since basic sets form a topological basis of $C_x$, it is enough to show that 
\begin{equation}\label{E:outin} 
( h_{\bar\iota})_*(\mu_y\times \mu_z)\res U \preceq \mu_x,
\end{equation} 
for each basic for $x$ set $U$ and each $\bar\iota$ as in \eqref{E:seq}. Fix such $U$ and $\bar\iota$, and set 
\[
\mu_{U,{\bar\iota}} = (h_{\bar\iota})_*(\mu_y\times \mu_z)\res U.
\]
Unpacking further, we see that showing inequality \eqref{E:outin} boils down to showing that, 
for each compact subset $K$ of $U$ with $\mu_{U,{\bar\iota}}\big(K\big)>0$, 
there is some $j\in {\mathbb N}$ with $\mu^j_x(K)>0$. 
This translates into proving the following implication for each compact set $K\subseteq U$
\begin{equation}\label{E:impl}
L^2(\mu_{U,{\bar\iota}}\res K)\not= 0 \Longrightarrow L^2(\mu^j_x\res K)\not= 0, \hbox{ for some }j. 
\end{equation}

We view $L^2(\mu_{U,{\bar\iota}})$ as the underlying Hilbert space of the representation $\rho_{U, {\bar\iota}}$ that 
is equal to $\rho_x$, as in Section~\ref{Su:basic}, on $L^2(\mu_{U,{\bar\iota}})$. To prove \eqref{E:impl}, 
it is now natural to seek a direct connection between $\rho_{U, {\bar\iota}}$ and $\tau$. This connection 
comes in the form of a Hilbert space map  
\[
\Psi\colon L^2(\mu_{U,{\bar\iota}})\to H_0(\gamma)
\]
that is equivariant between $\rho_{U,{\bar\iota}}$ and $\sigma$, or rather in the form of its composition 
with the isomorphism $\Phi$ yielding a Hilbert space map 
\begin{equation}\label{E:psip} 
\Phi\circ\Psi\colon L^2(\mu_{U,{\bar\iota}})\to H_\tau 
\end{equation}
that is equivariant between $\rho_{U,{\bar\iota}}$ and $\tau$. Such a $\Psi$ is constructed in Section~\ref{S:more}. 
The work done in Sections~\ref{S:indl}, \ref{S:sesq}, and \ref{S:tens} forms a basis of our construction of $\Psi$. 
In this work, the ergodic theorem for $L^0(\lambda, {\mathbb T})$ is used. 

The connection given by the Hilbert space map \eqref{E:psip} is exploited to get \eqref{E:impl} as follows. 
Given an arbitrary unitary representation $\xi$ of $L^0(\lambda, {\mathbb T})$ on a Hilbert space $H$, we associate 
with $x\in {\mathbb N}[{\mathbb Z}^\times]$ and a compact set $K\subseteq C^0_x$ a subspace of $H$, which we call 
$[x,K]^\xi$. The definition of and the results on $[x,K]^\xi$ are given in Section~\ref{S:subs}. 
It turns out that a Hilbert space map that is equivariant between two unitary representations $\xi_1$ and $\xi_2$  of 
$L^0(\lambda, {\mathbb T})$
maps the space $[x,K]^{\xi_1}$ for the first representation to a subspace of $[x,K]^{\xi_2}$ for the second representation. 
In particular, since Hilbert space maps are embeddings, we have 
\begin{equation}\label{E:dage}
[x,K]^{\xi_1}\not= 0 \Longrightarrow [x,K]^{\xi_2}\not= 0. 
\end{equation} 
Next, we need 
information on the spaces $[x,K]^\xi$ for $\xi= \rho_{U,{\bar \iota}}$  and $\xi=\tau$. We show that 
$[x,K]^{\rho_{U,{\bar \iota}}}$ contains $L^2(\mu_{U, {\bar\iota}}\res K)$, 
while $[x,K]^{\tau}$ is contained in the $\ell^2$-sum of the spaces $L^2(\mu^j_x\res \widetilde{K})$ over all $j\in {\mathbb N}$, 
where $\widetilde{K}$ is a symmetrization of $K$ using good homeomorphisms of $C_x$. 
From these inclusions, together with the general implication \eqref{E:dage} and the existence of the equivariant 
Hilbert space map \eqref{E:psip}, 
implication \eqref{E:impl} follows. This final argument is carried out in detail in Section~\ref{S:finp}.

\section{The proof of Theorem~\ref{T:main} from Theorem~\ref{T:notL}}

\centerline{{\bf From this point on, we write $L^0$ for $L^0(\lambda, {\mathbb T})$.}}

\smallskip

We start our proof with a lemma.

\begin{lemma}\label{L:orto}
Assume that there is a non-meager set of transformations $T\in {\rm Aut}(\gamma)$ such that 
$\langle T\rangle_c$ is included the image of a continuous homomorphism from 
$L^0(\nu, {\mathbb T})$ to ${\rm Aut}(\gamma)$ for some finite Borel measure $\nu$, with $\nu$ possibly depending on $T$. 
Then there exists an ergodic boolean action of $L^0$ on 
$(X,\gamma)$, whose Koopman representation is such that 
\begin{equation}\label{E:rtt}
\mu_x \otimes \mu_x \perp \mu_{x\oplus x}, \hbox{ for all } x\in {\mathbb N}[{\mathbb Z}^\times].
\end{equation}
\end{lemma}

\begin{proof} We state the relevant properties of the group $\langle T\rangle_c$ for a generic $T\in {\rm Aut}(\gamma)$. 
\begin{enumerate}
\item[(a)] The set 
\[
\{ S\in \langle T\rangle_c\mid S\hbox{ fulfills the condition in Theorem~\ref{T:spge}}\}
\]
is comeager in $\langle T\rangle_c$. 

\item[(b)] $\langle T\rangle_c$ is the largest abelian subgroup of ${\rm Aut}(\gamma)$ containing $T$. 

\item[(c)] There is no non-trivial continuous homomorphism from $\langle T\rangle_c$ to $\mathbb T$. 
\end{enumerate} 

\noindent Point (a) follows from Theorem~\ref{T:spge} and Lemma~\ref{L:kuth} as soon as we see that the set defined by the condition in 
Theorem~\ref{T:spge} has the Baire property in ${\rm Aut}(\gamma)$. 
But by \cite[Lemma~1.4]{DJL} this set is actually Borel, in fact, a countable intersection of open subsets of ${\rm Aut}(\gamma)$.  
Point (b) is an immediate consequence of 
Theorem~\ref{T:geco}\,(i). As for point (c), a non-trivial 
continuous homomorphism from $\langle T\rangle_c$ to $\mathbb T$ would induce a continuous action of 
$\langle T\rangle_c$ on $\mathbb T$ that does not have fixed points, contradicting Theorem~\ref{T:geco}\,(ii).

By the classical theorem of Halmos, see \cite[Theorem~2.6]{Kec2}, a generic $T\in {\rm Aut}(\gamma)$ is ergodic. 
Therefore, our assumption allows us 
to find an ergodic $T\in {\rm Aut}(\gamma)$ that fulfills (a--c) above and for which there exists 
a continuous homomorphism $\zeta\colon L^0(\nu, {\mathbb T})\to {\rm Aut}(\gamma)$, for some finite Borel measure $\nu$, 
with 
\begin{equation}\label{E:tin}
\langle T\rangle_c \subseteq \zeta\big(L^0(\nu, {\mathbb T})\big). 
\end{equation} 
We fix such a transformation $T$.

The inclusion \eqref{E:tin} together with (b) gives
\[
\langle T\rangle_c = \zeta\big(L^0(\nu, {\mathbb T})\big). 
\]
As a consequence 
\[
\zeta\colon L^0(\nu, {\mathbb T}) \to \langle T\rangle_c
\]
is a continuous surjective homomorphism between Polish groups, which makes it, by \cite[Theorem~2.3.3]{Gao}, an open map. Thus, 
$\langle T\rangle_c$ is isomorphic as a topological group to the quotient group 
\begin{equation}\label{E:isot}
\langle T\rangle_c \cong L^0(\nu, {\mathbb T})/ {\rm ker}(\zeta). 
\end{equation} 
Represent $\nu$ as the sum $\nu_1+\nu_2$, where $\nu_1$ is atomless and $\nu_2$ is purely atomic. 
Then $L^0(\nu, {\mathbb T})$ is isomorphic as a topological group to the product 
\begin{equation}\label{E:lopr}
L^0(\nu, {\mathbb T}) = L^0(\nu_1, {\mathbb T}) \times L^0(\nu_2, {\mathbb T}). 
\end{equation} 
With \eqref{E:lopr} in mind, note that if $\{ 1\}\times L^0(\nu_2, {\mathbb T})$ is not included in ${\rm ker}(\zeta)$, 
then there is a non-trivial continuous homomorphism from $L^0(\nu, {\mathbb T})/ {\rm ker}(\zeta)$ to $\mathbb T$, 
contradicting (c) in view of \eqref{E:isot}. Thus, 
\[
\{ 1\}\times L^0(\nu_2, {\mathbb T}) < {\rm ker}(\zeta). 
\]
It follows that $\zeta$ factors through a continuous surjective homomorphism 
\[
\zeta'\colon  L^0(\nu_1, {\mathbb T}) \to \langle T\rangle_c. 
\]
Since, by the ergodicity of $T$, the group $\langle T\rangle_c$ is non-trivial, we have that $\nu_1$ is non-zero, and, 
therefore, one can assume that $\nu_1=\lambda$.

It follows from the considerations above that there exists a continuous surjective homomorphism 
\[
\zeta\colon L^0\to \langle T\rangle_c\subseteq {\rm Aut}(\gamma). 
\]
Ergodicity of the boolean action $\zeta$ is an immediate consequence of ergodicity of $T$. 
Using openness of $\zeta$, we see that the preimage under $\zeta$ of the set in point (a) is comeager in 
$L^0$. Therefore, condition (i) of 
Theorem~\ref{T:dlco} is fulfilled. It follows from Theorem~\ref{T:dlco} that 
the Koopman representation associated with the boolean action $\zeta$ 
fulfills condition (ii) of that theorem. Taking $p=2$, $p'=1$, $\ell_1=\ell_2=\ell_1'=1$, $x_1=x_2=x$, and 
$x_1' = x\oplus x$ in condition (ii), we see that the Koopman representation fulfills \eqref{E:rtt}. 
\end{proof}

\begin{proof}[Proof of Theorem~\ref{T:main} from Theorem~\ref{T:notL}] 
Assume, towards a contradiction, that for a non-meager set of $T\in {\rm Aut}(\gamma)$, 
$T$ is in the image of a continuous homomorphism from $L^0(\nu, {\mathbb T})$ to ${\rm Aut}(\gamma)$ for 
some finite Borel measure $\nu$. Then, by Lemma~\ref{L:orto}, there is 
a boolean action of $L^0$, which is ergodic and whose Koopman representation fulfills \eqref{E:rtt}. 
On the other hand, by Theorem~\ref{T:notL}, for ergodic boolean actions of $L^0$, we have 
\begin{equation}\label{E:sst}
\mu_x\otimes \mu_x \preceq \mu_{x\oplus x}, \hbox{ for all } x\in {\mathbb N}[{\mathbb Z}^\times].
\end{equation} 
Now \eqref{E:rtt} and \eqref{E:sst} give $\mu_x=0$, for all $x\in {\mathbb N}[{\mathbb Z}^\times]$. 
Thus, the boolean action of $L^0$ is trivial, making it not ergodic, a contradiction. 
\end{proof}

\section{Ergodic theorem for $L^0(\lambda, {\mathbb T})$} 

Theorem~\ref{T:ert} below plays an auxiliary, but important, role in the proof of Theorem~\ref{T:notL}. 
It is an analogue for ergodic boolean actions of $L^0$ of the ergodic theorem. 
It concerns the sequences \eqref{E:opa} defined below. 

Boolean and near-actions 
are defined in Section~\ref{S:thou}; groups ${\mathbb S}_n$ and notation concerning these groups 
are defined in Section~\ref{Su:twno}.

For measurability considerations, we will need the following lemma. 
We will use it a couple of times. 

\begin{lemma}\label{L:help}
Let $H$ be a compact metric group equipped with the 
probability Haar measure $\theta$. 
Let $\zeta\colon H\to {\rm Aut}(\gamma)$ be a boolean action of $H$. 
Then a near action $H\times X\to X$ inducing $\zeta$ is measure preserving 
between $(H\times X, \theta\times \gamma)$ and $(X,\gamma)$, 
that is, for each measurable set $A\subseteq X$, 
\[
(\theta\times \gamma)\big( \{ (h,x)\in H\times X\mid hx\in A\} )\big) = \gamma(A). 
\]
\end{lemma}

\begin{proof} By Fubini's theorem, we have 
\[
\begin{split}
(\theta\times \gamma)\big( \{ (h,x)\in H\times X\mid hx\in A\} )\big) = \int_H \gamma(h^{-1}A)\, d\theta(h) = \int_H \gamma(A)\, d\theta =\gamma(A),
\end{split} 
\]
and the lemma follows. 
\end{proof}

Assume we have a boolean action of $L^0$ on $(X,\gamma)$. Fix $n\in {\mathbb N}$. 
By Lemma~\ref{L:help}, applied to $H={\mathbb S}_n$, and Fubini's theorem, for each $f\in L^1(\gamma)$, the function 
\[
{\mathbb S}_n\ni t\to f(t\omega) 
\]
is in $L^1(\theta)$ for almost all $\omega\in X$ with respect to $\gamma$. This observation allows us 
to define, for almost all $\omega\in X$
with respect to $\gamma$, 
\begin{equation}\label{E:opa}
(A_nf)(\omega) = \int_{{\mathbb S}_n}f(t \omega)\,d\theta(t).
\end{equation} 
Furthermore, $A_nf$ is in $L^1(\gamma)$. 

We can now state the ergodic theorem for $L^0$. 
Our original proof of it (with $f\in L^2(\gamma)$ and the convergence in the conclusion being in $L^2(\gamma)$, 
which is sufficient for our applications) 
used  the main result from \cite{Sol} restated above as Theorem~\ref{T:rep}. 
The proof below was pointed out to us by Glasner and Weiss, and we reproduce it here with their permission.

\begin{theorem}\label{T:ert} 
Let an ergodic boolean action of $L^0$ on $(X, \gamma)$ be given. For $f\in L^1(\gamma)$, 
the sequence $(A_nf)_n$ converges pointwise almost everywhere with respect to $\gamma$ to the function 
constantly equal to $\int_X f\,d\gamma$. 
\end{theorem}

\begin{proof}
For $n\in {\mathbb N}$, define ${\mathcal F}_{-n}$ 
to consist of all $\gamma$-measurable subsets  $A$ of $X$ such that 
$\gamma(tA\setminus A)=0$ for all $t\in {\mathbb S}_n$. 
It is clear that ${\mathcal F}_{-n}$ is a $\sigma$-algebra and that ${\mathcal F}_{-n}\supseteq {\mathcal F}_{-(n+1)}$. Set 
${\mathcal F}_{-\infty} = \bigcap_n {\mathcal F}_{-n}$. 
By the backwards martingale theorem \cite[Theorem~4.7.3]{Dur}, we have 
\begin{equation}\label{E:maco}
E\big(f|{\mathcal F}_{-n}\big) \to E\big(f|{\mathcal F}_{-\infty}\big)\hbox{ as }n\to \infty,
\end{equation}
where the convergence is taken to be pointwise $\gamma$-almost everywhere. 

By the ergodicity of the boolean action and the density of $\bigcup_n {\mathbb S}_n$ in 
$L^0$, the $\sigma$-algebra 
${\mathcal F}_{-\infty}$ consists of sets of $\gamma$-measure $0$ and of full $\gamma$-measure; thus, 
\[
E\big(f|{\mathcal F}_{-\infty}\big) = \int_X f\,d\gamma.
\]
From the invariance under translations by elements of ${\mathbb S}_n$ 
of the Haar measure $\theta$, $A_nf$ is invariant under the boolean action of ${\mathbb S}_n$, 
so it is measurable with respect to ${\mathcal F}_{-n}$.
This observation, invariance of $\gamma$ under the boolean action of ${\mathbb S}_n$, and Fubini's theorem imply that 
\[
E\big(f|{\mathcal F}_{-n}\big)  = A_nf.
\]
The conclusion follows from \eqref{E:maco}. 
\end{proof}

\section{A subspace for a unitary representation of $L^0(\lambda, {\mathbb T})$}\label{S:subs} 

We start working towards the proof of Theorem~\ref{T:notL}. 
The notion of equivariant Hilbert space map can be found Section~\ref{Su:equi}.

Fix $x\in {\mathbb N}[{\mathbb Z}^\times]$ and a compact set $K\subseteq C^0_x$; they will remain fixed for 
the rest of this section.

Let $\xi$ be a unitary representation of $L^0$ on a Hilbert space $H$. 
We describe here a subspace $[x,K]^\xi$ of $H$ associated 
with a point $x\in {\mathbb N}[{\mathbb Z}^\times]$ and a compact set $K\subseteq C^0_x$. 
In following the definition of this space, it may be useful to keep in mind Lemma~\ref{L:pfco}. An analysis of spaces of this form 
will be crucial in our proof of Theorem~\ref{T:notL}. 
Define the subset
\begin{equation}\label{E:xkx}
[x,K]^\xi
\end{equation}  
of $H$ as follows. A number $n\in {\mathbb N}$ will be called
{\bf admissible} if 
\[
K\subseteq \bigcup_u \llbracket u\rrbracket,
\]
where $u$ varies over the set of all injections $u\colon D(x) \to 2^n$. 
For an admissible $n$, put
\[
P_{n}(K) = \{ u\mid u\colon D(x) \to 2^n \hbox{ an injection and } K\cap \llbracket u\rrbracket \not=\emptyset\} .
\]
Note that $P_n(K)$ is finite. Using compactness of $K\subseteq C^0_x$,
we see that all large enough $n\in {\mathbb N}$ are admissible. 
With a sequence $(z_s\colon s\in 2^n)$, with $z_s\in {\mathbb T}$ for all $s\in 2^n$, we associate the element 
$\sum_{s\in 2^n} z_s\chi_{[s]}$ of $L^0$, which gives the unitary operator
\begin{equation}\label{E:bigf}
\xi(\sum_{s\in 2^n} z_s\chi_{[s]}).
\end{equation}
Define $[x,K]^\xi$ to be the set of all $h\in H$ with the
following property. For every admissible $n\in {\mathbb N}$, $h$ can be represented as
\begin{equation}\label{E:bigs}
h = \sum_{u\in P_n(K)}h_u,
\end{equation}
where, for every $u\in P_n(K)$ and $(z_s\colon s\in 2^n)$, $h_u$ belongs to the eigenspace associated with 
the eigenvalue
\begin{equation}\label{E:bige}
\prod_{(k,i)\in D(x)} z_{u(k,i)}^{k}
\end{equation}
of the operator \eqref{E:bigf}. Since eigenspaces are linear spaces, it follows easily that $[x,K]^\xi$ is a linear subspace of 
$H$, but we will not use this fact.

The following simple, but useful, lemma makes explicit a degree invariance of the space defined above. 

\begin{lemma}\label{L:kpr}  
Let $\xi_1$ and $\xi_2$ be unitary representations of $L^0$ on Hilbert spaces $H_1$ and $H_2$, respectively. Let 
$\Gamma\colon H_1\to H_2$ be a Hilbert space map that is equivariant between $\xi_1$ and $\xi_2$. Then 
\[
\Gamma\big( [x,K]^{\xi_1}\big)\subseteq [x,K]^{\xi_2}.
\]
\end{lemma}

\begin{proof} It suffices to notice two points. First, $\Gamma$ is linear being a Hilbert space map. 
Second, if $h\in H_1$ is in the eigenspace associated with eigenvalue $c\in {\mathbb T}$ of 
the operator $\xi_1(\phi)$, for some $\phi\in L^0$, then, by equivariance and linearity of $\Gamma$, the vector 
$\Gamma(h)\in H_2$ is in the eigenspace associated with $c$ of the operator $\xi_2(\phi)$.
\end{proof}

The next two lemmas give estimates on the size of the space $[x,K]^\xi$. 
In both lemmas, one can actually prove equalities in place of the indicated
inclusions but, in the sequel, we will only need the inclusions. 

For a set $A\subseteq C_x$, let 
\begin{equation}\label{E:goar}
\widetilde{A} = \bigcup_{\delta}\widetilde{\delta}(A), 
\end{equation} 
where $\delta$ ranges over all good permutations of $D(x)$, and so $\widetilde{\delta}$ ranges 
over all good homeomorphisms of $C_x$; see Section~\ref{Suu:cx} for the definitions of good permutations and good 
homeomorphisms. 

\begin{lemma}\label{L:co2}
Let $\xi$ be a unitary representation of $L^0$. Let $\mu^j_x$, for $j\in {\mathbb N}$ and $x\in {\mathbb N}[{\mathbb Z}^\times]$,  
be measures found for $\xi$ by Theorem~\ref{T:rep}.

Fix $x\in {\mathbb N}[{\mathbb Z}^\times]$ and a compact set $K\subseteq C^0_x$. 
The space $[x, K]^\xi$ 
is a subspace of the $\ell^2$-sum over $j\in {\mathbb N}$ of the spaces $\widetilde{L^2}(\mu^j_x\res \widetilde{K})$. 
\end{lemma}

\begin{proof} This proof is a modification of the argument in \cite[p.\,3118]{Sol}. 
We start with an elementary claim, whose justification we leave to the reader.

\setcounter{claim}{0}

\begin{claim}\label{Cl:1}
Let $a, b, c$ be finite sets with $a, b\subseteq c$, and 
let $(l_\nu)_{\nu\in a}$ and $(m_\nu)_{\nu\in b}$ be sequences of elements of $\mathbb N$.
Assume that, for each sequence $(z_\nu)_{\nu\in c}$ of elements of $\mathbb T$, we have 
\[
\prod_{\nu\in a} z_\nu^{l_\nu}=  \prod_{\nu\in b} z_\nu^{m_\nu}.
\]
Then, for each $r\in {\mathbb N}$, 
\[
\{ \nu\in a\mid l_\nu = r\} = \{ \nu\in b\mid m_\nu = r\}.
\]
\end{claim}

Assume towards a contradiction that the inclusion in the conclusion of the lemma does not hold. 
This means that there is an element
$h$ of $[x,K]^\tau$ such that one of the following two possibilities occurs: 
\begin{enumerate}
\item[(a)] for some $x'\not= x$ and some $j$, the orthogonal projection of $h$ on $\widetilde{L^2}(\mu^j_{x'})$ is
non-zero;

\item[(b)]  for some $j$, the orthogonal projection of $h$ on $\widetilde{L^2}(\mu^j_{x})$ has
support not included in $\widetilde{K}$.
\end{enumerate} 

Fix such an element $h$, and let $A$ be the support of the projection of $h$ as in (a) or (b) above. 
So $A$ is a non-zero $\mu^j_{x'}$-measure 
class of a Borel subset of $C^0_{x'}$ in case (a), and it is non-zero $\mu^j_{x}$-measure class of a Borel subset of $C^0_{x}$ in case (b). 
We view $A$ 
as a Borel set keeping in mind that it is determined by $h$ only up to a measure zero set. With this convention, we make the following claim that combines (a) and (b). 

\begin{claim}\label{Cl:2} 
There are $x'$, $j$, $n$, and an injection $v\colon D(x') \to 2^n$ such that 
$n$ is admissible for $K$, 
\begin{equation}\label{E:possu}
\mu^j_{x'}\big( A\cap \llbracket v\rrbracket \big) >0,
\end{equation}
and either $x'\not= x$ or ($x'=x$ and  $v\circ\delta\not= u$ for all $u\in P_n(K)$ and all good permutations $\delta$ of $D(x)$). 
\end{claim} 

\noindent {\em Proof of Claim~\ref{Cl:2}.} As already remarked, large enough $n$ are admissible for $K$ by compactness of $K$ 
and the inclusion $K\subseteq C^0_x$. 

If (a) holds, we pick $x'$ and $j$ as in (a). Since by Lemma~\ref{L:basi} sets of the form $\llbracket v\rrbracket$,  
for injections $v\colon D(x') \to 2^n$, for large enough $n$, form a topological basis for $C^0_{x'}$, we can find such an injection 
$v$ with \eqref{E:possu} for each large enough $n$. 

If (b) holds, then, using compactness of $K$ and Lemma~\ref{L:basi} again, 
for each large enough $n$, we can find an injection $v\colon D(x') \to 2^n$ such that 
\eqref{E:possu} holds and we have $\llbracket v\rrbracket \cap {\widetilde K}=\emptyset$. This disjointness condition translates to 
$\llbracket v\circ \delta \rrbracket \cap K =\emptyset$, for each good permutations $\delta$ of $D(x)$, which gives 
$v\circ \delta\not\in P_n(K)$, for each good permutations $\delta$ of $D(x)$, and the claim follows. 

\medskip

Fix $x'$, $j$, $n$, and $v$ as in Claim~\ref{Cl:2}. 
Now, since $n$ is admissible for $K$, $h$ is represented as a sum as in \eqref{E:bigs} for this $n$. 
By \eqref{E:possu}, there exists an $h_u$, for some $u\in P_n(K)$, whose orthogonal projection on
$L^2(\mu^{j}_{x'})$ has support intersecting $\llbracket v\rrbracket$ on a set of positive measure with respect to $\mu^j_{x'}$. 
Fix such $u$. 
Since 
$h_u$ is non-zero, it is an eigenvector of the operator in 
\eqref{E:bigf} for each sequence $(z_s\colon s\in 2^n)$. Its eigenvalue for a given $(z_s\colon s\in 2^n)$ must be
equal to
\begin{equation}\label{E:eigc} 
\prod_{(k,i)\in D(x')} z_{v(k,i)}^{k}.
\end{equation}
since, by Lemma~\ref{L:pfco}, every value of each function from $\widetilde{L^2}(\mu^{j}_{x'})$
attained on $\llbracket v\rrbracket$ is multiplied by that
number when the function is acted on by the operator in \eqref{E:bigf}. On the other hand, this
eigenvalue is also equal to \eqref{E:bige} for the $u$ found above. Thus, \eqref{E:bige} and \eqref{E:eigc} are equal to 
each other for each choice of $(z_s\colon s\in 2^n)$. 
Using Claim~\ref{Cl:1} (with $a= u\big(D(x)\big)$, $b=v\big(D(x')\big)$, $c=2^n$, $l_{u(k,i)}= k$, and $m_{v(k,i)}=k$), 
we see that $x'=x$ 
and $v\circ\delta =u$, for some good permutation $\delta$, which is a contradiction.
\end{proof}

\setcounter{claim}{0}

Lemma~\ref{L:co1} below gives a lower estimate on the space $[x,K]^\xi$ for a representation $\xi$ of one 
of the two types defined in Section~\ref{Su:basic}. 

\begin{lemma}\label{L:co1}
Fix $x\in {\mathbb N}[{\mathbb Z}^\times]$ and a compact set $K\subseteq C^0_x$. Let $\mu$ be 
a measure marginally compatible with $x$ that is concentrated on $K$. Let $\xi$ be the representation 
equal to $\rho_x$ on $L^2(\mu)$. Then 
\[
L^2(\mu)\subseteq [x,K]^\xi.
\]
\end{lemma}

\begin{proof}
Let $h\in L^2(\mu)$. Fix $n$ that is admissible for $K$, and set 
\[
h_u =  h\res \llbracket u\rrbracket,
\]
for $u\in P_n(K)$. Since $\mu$ concentrates on $K$ and 
\[
K\subseteq \bigcup_{u\in P_n(K)} \llbracket u\rrbracket, 
\]
we see that in $L^2(\mu)$
\begin{equation}\label{E:hdec}
h = \sum_{u\in P_n(K)}h_u.
\end{equation}
Now, take a sequence $(z_s\colon s\in 2^n)$, with $z_s\in {\mathbb T}$, for all $s\in 2^n$, and consider 
the associated with the sequence element 
\[
t= \sum_{s\in 2^n} z_s\chi_{[s]}\in L^0. 
\]
Note that, by the definitions of $\xi$ and $h_u$ and by Lemma~\ref{L:pfco}, we get 
\[
\xi(t)\big(h_u\big) = \big(\prod_{(k,i)\in D(x)} z_{u(k,i)}^{k}\big) h_u.
\]
This equation shows that $h_u$ belongs to the appropriate eigenspace, and the lemma follows by \eqref{E:hdec}. 
\end{proof}

\section{A lemma on independence of random variables}\label{S:indl}  

We assume we are given a boolean action of $L^0$ on a Borel probability space $(X,\gamma)$. 
We fix notation for two unitary representations of $L^0$ associated with the boolean action. 

{\bf Representation $\sigma$.} This is the Koopman representation on $L^2(\gamma)$ induced by the boolean action of $L^0$, that is, 
for $\phi\in L^0$, $\sigma(\phi)$ is the unitary operator on $L^2(\gamma)$, whose value on $f\in L^2(\gamma)$ is given by 
\[
\big(\sigma(\phi)(f)\big)(\omega) = f(\phi^{-1}\omega), \hbox{ for }\omega\in X. 
\]
Let 
\[
H_0(\gamma) = \{ f\in L^2(\gamma)\mid \int_X f\,d\gamma=0\}. 
\]
Then $H_0(\gamma)$ is a closed subspace of $L^2(\gamma)$ that is invariant under $\sigma$. 
Note that by ergodicity of the boolean action of $L^0$ the space $(X,\gamma)$, we have 
\begin{equation}\label{E:gain}
H_0(\gamma) = \big( \{ f\in L^2(\gamma)\mid \sigma(\phi)(f) = f\hbox{ for all }\phi\in L^0\}\big)^\perp. 
\end{equation}

{\bf Representation $\tau$.} This is the representation obtained by applying Theorem~\ref{T:rep} to 
the representation $\sigma$ above. 
So $\tau$ is determined by a sequence of finite Borel measures 
$(\mu^j_x)_{x\in N[{\mathbb Z}^\times], j\in {\mathbb N}}$ such that, for each $x$ and $j$, 
\[
\mu_x^j\hbox{ is compatible with $x$, and }  \mu^{j+1}_x\preceq \mu^j_x. 
\]
The underlying Hilbert space of $\tau$ is the $\ell^2$-sum of the Hilbert spaces $\widetilde{L^2}(\mu^j_x)$, with the sum taken 
over $x\in {\mathbb N}[{\mathbb Z}^\times]$ and $j\in {\mathbb N}$. 
We denote this Hilbert space by $H_\tau$. Representation $\tau$ on $H_\tau$ is then given by formulas 
\eqref{E:pff}. 

By Theorem~\ref{T:rep} and \eqref{E:gain}, there is a Hilbert space isomorphism 
\begin{equation}\label{E:hot}
\Phi\colon H_0(\gamma) \to H_\tau
\end{equation}
that is equivariant between $\sigma$ and $\tau$. For $f\in H_\tau$, we let 
\begin{equation}\label{E:hatt}
{\hat f} = \Phi^{-1}(f). 
\end{equation}

The proof of Lemma~\ref{L:inde} below uses the ergodic theorem for boolean actions of $L^0$, Theorem~\ref{T:ert}. 
We will need the following notion. 
Let $M$ be a set of complex valued functions defined on the same set. By the {\bf $*$-algebra generated by $M$} we understand 
the set of all functions obtained by closing $M$ under addition, multiplication, conjugation, and multiplication by complex scalars. 
Additionally, we adopt the following convention. We write ${\hat f}$ for $\hat {\tilde f}$, 
for $f\in L^2(\mu_y\res \llbracket v\rrbracket )$ with $y\in {\mathbb N}[{\mathbb Z}^\times]$ 
and $v\colon D(y)\to 2^n$,
where $\tilde f$ is as in Lemma~\ref{L:ident} for $U=\llbracket v\rrbracket$.

\begin{lemma}\label{L:inde} 
Let $y_1, \dots, y_k\in {\mathbb N}[{\mathbb Z}^\times]$. For $1\leq i\leq k$, 
let $v_i\colon D(y_i)\to 2^n$ be an injection  and let $F_i$ be a function in the $*$-algebra generated by the set 
\[
\{ \hat{f}\mid f\in L^2(\mu_{y_i}\res \llbracket v_i\rrbracket)\}.
\]
If, for all $1\leq i<j\leq k$, the images of $v_i$ and $v_j$ are disjoint, 
then $F_1, \dots , F_k$ are independent random variables on $(X,\gamma)$. 
\end{lemma}

\begin{proof}To keep our notation light, we present the proof for $k=2$, and we set 
$F=F_1$, $G=F_2$, $v=v_1$, and $w=v_2$. 

Let $A, B\subseteq {\mathbb C}$ be Borel sets. Our aim is to show that 
\begin{equation}\notag
\gamma\big( F^{-1}(A) \cap G^{-1}(B)\big) = \gamma\big( F^{-1}(A)\big) \,\gamma\big( G^{-1}(B)\big). 
\end{equation} 
If we let $\chi^1$, $\chi^2$, and $\chi^{12}$ be the indicator functions of the sets 
\[
F^{-1}(A),\, G^{-1}(B),\hbox{ and }  F^{-1}(A) \cap G^{-1}(B),
\]
respectively, then we need to prove 
\begin{equation}\label{E:inf}
\int_X \chi^{12}\,d\gamma = \int_X \chi^1\,d\gamma\,\int_X\chi^2\,d\gamma. 
\end{equation}

Let $p\geq n$, where $n$ is as in the statement of the lemma. We associate with $v$ and $w$ sets 
of finite sequences in $2^p$ as follows 
\[
\begin{split}
(v,p) &= \{ s\in 2^p\mid v(k,i)= s\res n,\hbox{ for some }(k,i)\in D(y)\}\\
(w,p) &= \{ s\in 2^p\mid w(k,i)= s\res n, \hbox{ for some }(k,i)\in D(z)\}. 
\end{split} 
\]
Recall that ${\mathbb S}_p = {\mathbb T}^{2^p}$. Define the following closed subgroups of ${\mathbb S}_p$ 
\[
\begin{split}
S_{v,p} &= \{ t\in {\mathbb S}_p\mid t(s)=1 \hbox{ for all } s\in (v,p)\},\\
S_{w,p} &= \{ t\in {\mathbb S}_p\mid t(s)=1 \hbox{ for all } s\in (w,p)\},\\
S_{v,p}^\perp &= \{ t\in {\mathbb S}_p\mid t(s)=1 \hbox{ for all } s\in 2^p\setminus (v,p)\},\\
S_{w,p}^\perp &= \{ t\in {\mathbb S}_p\mid t(s)=1 \hbox{ for all } s\in 2^p\setminus (w,p)\}. 
\end{split}
\]
While, as usual, $\theta$ is the probability Haar measure on ${\mathbb S}_p$, additionally, we let 
\[
\theta_v,\, \theta_w,\, \theta_v^\perp,\, \theta_w^\perp 
\]
be the probability Haar measures on $S_{v,p}$, $S_{w,p}$, $S_{v,p}^\perp$, $S_{w,p}^\perp$, respectively. 

We record an identity, contained in \eqref{E:ide2}, 
that is at the root of the proof of independence of $F$ and $G$. 
Let $A'\subseteq S_{v,p}^\perp$ and $B'\subseteq S_{w,p}^\perp$ be Borel sets. 
First, we view 
${\mathbb S}_p$, both as a topological group and as a measure space 
equipped with $\theta$, as products 
\[
{\mathbb S}_p = S_{v,p}^\perp \times S_{v,p}\;\hbox{ and }\; {\mathbb S}_p = S_{w,p}^\perp \times S_{w,p},
\]
with the groups on the right hand sides equipped with the measures $\theta_v^\perp\times \theta_v$ and 
$\theta_w^\perp\times \theta_w$, respectively. Thus, we have 
\begin{equation}\label{E:ide}
\theta\big(A'\cdot S_{v,p} \big)=\theta_v^\perp(A')\;\hbox{ and }\;
\theta\big(B' \cdot S_{w,p} \big)=\theta_w^\perp(B'). 
\end{equation} 
Further, since the images of the injections $v$ and $w$ are disjoint, we have 
\[
(v,p)\cap (w,p)=\emptyset, 
\]
which allows us to view ${\mathbb S}_p$ also as the product 
\[
{\mathbb S}_p = S_{v,p}^\perp \times S_{w,p}^\perp \times (S_{v,p} \cap S_{w,p}). 
\]
Therefore, we get 
\begin{equation}\notag
\begin{split}
(A' \cdot S_{v,p}) \cap (B' \cdot S_{w,p}) &= A'\cdot B' \cdot (S_{v,p}\cap S_{w,p})\;\hbox{ and }\\
\theta\big(A'\cdot B' \cdot (S_{v,p}\cap S_{w,p}) \big)&=\theta_v^\perp(A')\, \theta_w^\perp(B'). 
\end{split}
\end{equation} 
Putting the two above equalities and \eqref{E:ide} together, we have 
\begin{equation}\label{E:ide2} 
\theta\big( (A' \cdot S_{v,p}) \cap (B' \cdot S_{w,p})\big) 
= \theta\big(A'\cdot S_{v,p} \big)\,\theta\big(B' \cdot S_{w,p} \big).
\end{equation}

Fix $f\in L^2(\mu_y\res \llbracket v\rrbracket)$. 
By 
formula \eqref{E:pff} that defines representation $\tau$ and by Lemma~\ref{L:pfco}, we see that 
\begin{equation}\label{E:noth1}
\tau(t)\big( f\big) = f, \hbox{ for all }t\in S_{v,p}. 
\end{equation}
Since $\Phi$ is equivariant between $\sigma$ and $\tau$, by formula \eqref{E:hatt}, equation \eqref{E:noth1}
implies  
\[
\sigma(t)\big( {\hat f}\big) = {\hat f}, \hbox{ for all }t\in S_{v,p}.
\]
The above equation, the definition of $\sigma$, and the fact that $S_{v,p}$ 
is a group, imply that,  
for each $t\in S_{v,p}$, 
\[
{\hat f}(t\omega) = {\hat f}(\omega), \hbox{ for $\gamma$-almost every $\omega\in X$.}
\]

Applying Lemma~\ref{L:help} with $H= {\mathbb S}_p$ to the above condition, we get that, for each $t\in S_{v,p}$, 
\[
{\hat f}(tt'\omega) = {\hat f}(t'\omega), \hbox{ for $(\gamma\times \theta)$-almost every $(t', \omega)\in {\mathbb S}_p\times X$.}
\]
After applying Fubini's theorem, the condition above gives that, 
for $\gamma$-almost every $\omega\in X$, for $\theta$-almost every $t'\in {\mathbb S}_p$, we have 
\begin{equation}\label{E:cons0}
{\hat f}(tt'\omega) = {\hat f}(t'\omega), \hbox{ for $\theta_v$-almost every }t\in S_{v,p}.
\end{equation}
We proved that condition \eqref{E:cons0} holds for arbitrary 
$f\in L^2\big(\mu_y\res \llbracket v\rrbracket\big)$.
Since this condition persists under taking sums and products, conjugation, 
and multiplication by scalars of functions in $\{ \hat{f}\mid f\in L^2\big(\mu_y\res \llbracket v\rrbracket\big)\}$, we see 
that the following condition holds, for $\gamma$-almost every $\omega\in X$: 
\begin{equation}\label{E:cons}
\begin{split}
\hbox{for }\theta&\hbox{-almost every } t'\in {\mathbb S}_p,\\
F(tt'&\omega) = F(t'\omega), \hbox{ for $\theta_v$-almost every }t\in S_{v,p}. 
\end{split} 
\end{equation}

By the same argument, we get that 
 for $\gamma$-almost every $\omega\in X$, we have 
\begin{equation}\label{E:cons2}
\begin{split}
\hbox{for }\theta&\hbox{-almost every } t'\in {\mathbb S}_p,\\
G(tt'&\omega) = G(t'\omega), \hbox{ for $\theta_w$-almost every }t\in S_{w,p}. 
\end{split} 
\end{equation}

Fix $\omega\in X$ for which both \eqref{E:cons} and \eqref{E:cons2} hold. The set of such points $\omega\in X$ 
has $\gamma$-measure $1$. 
Consider the functions $F_{p,\omega}, G_{p,\omega}\colon {\mathbb S}_p\to {\mathbb C}$ given by 
\[
F_{p,\omega}(t')= F(t'\omega)\;\hbox{ and }\; G_{p,\omega}(t')= G(t'\omega). 
\]
With these definitions, condition \eqref{E:cons} translates to 
\begin{equation}\notag
\begin{split}
\hbox{for }\theta&\hbox{-almost every } t'\in {\mathbb S}_p,\\
F_{p,\omega}&(tt') = F_{p,\omega}(t'), \hbox{ for $\theta_v$-almost every }t\in S_{v,p}, 
\end{split} 
\end{equation}
We view $\theta$ as the product $\theta_v^\perp\times \theta_v$ and apply Fubini's theorem to the statement above and
obtain the following conclusion
\begin{equation}\notag
\begin{split}
\hbox{for }\theta_v^\perp&\hbox{-almost every } t' \in S^\perp_{v,p},\\
S_{v,p} \ni &\, t \to F_{p,\omega}(tt') \hbox{ is constant $\theta_v$-almost everywhere.} 
\end{split} 
\end{equation}
By the above observation, there exists a Borel set $A'\subseteq S^\perp_{v,p}$ such that 
\begin{equation}\label{E:prepr} 
F_{p,\omega}^{-1}(A) = A' \cdot S_{v,p} \hbox{ modulo a $\theta$-measure $0$ subset of }{\mathbb S}_p.
\end{equation} 
By the same argument, except that using \eqref{E:cons2} in place of \eqref{E:cons}, 
there exists a Borel set $B'\subseteq S^\perp_{w,p}$ such that 
\begin{equation}\label{E:prepr2} 
G_{p,\omega}^{-1}(B) = B' \cdot S_{w,p}\hbox{ modulo a $\theta$-measure $0$ subset of }{\mathbb S}_p.
\end{equation} 
By \eqref{E:prepr} and \eqref{E:prepr2}, using \eqref{E:ide2}, we get 
\begin{equation}\label{E:fieq} 
\theta\big( F_{p,\omega}^{-1}(A) \cap G_{p,\omega}^{-1}(B)\big) 
= \theta\big( F_{p,\omega}^{-1}(A)\big)\, \theta\big( G_{p,\omega}^{-1}(B)\big).
\end{equation} 

Let now $\chi^1_{p,\omega}$, $\chi^2_{p,\omega}$, and $\chi^{12}_{p,\omega}$ 
be the indicator functions of the following subsets of ${\mathbb S}_p$ 
\[
F_{p,\omega}^{-1}(A),\; G_{p,\omega}^{-1}(B), \hbox{ and } F_{p,\omega}^{-1}(A)\cap G_{p,\omega}^{-1}(B),
\]
respectively. Equation \eqref{E:fieq} gives 
\begin{equation}\label{E:gaga}
\int_{{\mathbb S}_p} \chi^{12}_{p,\omega}(t)\, d\theta(t ) = 
\int_{{\mathbb S}_p} \chi^1_{p,\omega}(t)\, d\theta(t)\, \int_{{\mathbb S}_p} \chi^2_{p,\omega}(t)\, d\theta(t).
\end{equation} 

We keep in mind that identity \eqref{E:gaga} was proved for all $p\geq n$ and for $\gamma$-almost all $\omega\in X$. 
Let $A_p$, for $p\in {\mathbb N}$, be the operator defined by \eqref{E:opa}. 
We observe that, for all $p$ and $\omega\in X$, 
\[
\int_{{\mathbb S}_p} \chi^1_{p,\omega}(t)\, d\theta(t)  = \int_{{\mathbb S}_p} \chi^1(t\omega) \,d\theta(t) = A_p(\chi^1)(\omega)
\]
and, similarly, 
\[
\int_{{\mathbb S}_p} \chi^2_{p,\omega}(t)\, d\theta(t) = A_p(\chi^2)(\omega)
\;\hbox{ and }\; \int_{{\mathbb S}_p} \chi^{12}_{p,\omega}(t)\, d\theta(t) = A_p(\chi^{12})(\omega).
\]
From the above and from \eqref{E:gaga}, we get that, for all $p\geq n$ and for $\gamma$-almost all points $\omega\in X$, 
\begin{equation}\label{E:prot}
A_p(\chi^{12})(\omega) = A_p(\chi^1)(\omega)\,A_p(\chi^2)(\omega).
\end{equation}

By Theorem~\ref{T:ert}, for $\gamma$-almost all $\omega\in X$, we have
\[
A_p(\chi^{12})(\omega) \to \int_X\chi^{12}\,d\gamma,\; 
A_p(\chi^{1})(\omega) \to \int_X\chi^{1}\,d\gamma,\hbox{ and }\,
A_p(\chi^{2})(\omega) \to \int_X\chi^{2}\,d\gamma,
\]
as $p\to\infty$. From the above and from \eqref{E:prot}, we get \eqref{E:inf} as required. 
\end{proof}

\section{A lemma on sesquilinear Hilbert space maps}\label{S:sesq}

Let $F_1, \dots, F_k$, $F_1', \dots, F_l'$, and $H$ be vector spaces over $\mathbb C$. Following 
\cite{BT}, we call a function 
\begin{equation}\label{E:sesm}
p\colon \prod_{i=1}^k F_i \times \prod_{i=1}^l F_i'\to H
\end{equation} 
a {\bf sesquilinear map} if it is linear in the coordinates $F_1, \dots, F_k$ and semilinear in the coordinates 
$F_1', \dots, F_l'$. A more precise name for such maps would be multi-sesquilinear, but we use the shorter name. Also 
the split into the linear and semilinear coordinates should be reflected in the name, but in all cases below the 
split will be evident from the notation and context. 

The following definitions generalize the notion of equivariant Hilbert space map. 
Assume now that the spaces $F_1, \dots, F_k$, $F_1', \dots, F_l'$, and $H$ are Hilbert spaces with inner products 
$\langle \cdot, \cdot\rangle_1, \dots, \langle \cdot, \cdot\rangle_k$, $\langle \cdot, \cdot\rangle'_1, \dots, \langle \cdot, \cdot\rangle'_l$, 
and 
$\langle \cdot, \cdot\rangle$, respectively. 
A function $p$ as in \eqref{E:sesm} is called a 
{\bf sesquilinear Hilbert space map} if it is sesquilinear and, 
for all $f_1, g_1\in F_1, \dots, f_k, g_k \in  F_k$ and $f'_1, g_1'\in F'_1, \dots, f'_l, g_l' \in  F'_l$, we have 
\[
\big\langle p(f_1, \dots, f_l'), p(g_1,\dots , g_l')\big\rangle  = 
\prod_{i=1}^k\langle f_i, g_i\rangle_i \,\prod_{i=1}^l \overline{ \langle f_i',  g_i'\rangle_i}.
\]
Let $\rho_1, \dots, \rho_k$, $\rho'_1, \dots, \rho'_l$, and $\xi$ be unitary representations of $L^0$ on the Hilbert spaces 
$F_1, \dots, F_k$, $F_1', \dots, F_l'$, and $H$, respectively. 
We say that $p$ is {\bf equivariant between $\big((\rho_i)_{i=1}^k, (\rho_i')_{i=1}^l\big)$ and $\xi$} if, for each $\phi\in L^0$ and 
$f_1\in F_1, \dots, f_k\in F_k$ and $f_1'\in F_1', \dots, f_l'\in F_l'$, we have 
\begin{equation}\label{E:peqi}
p\Big(\rho_1(\phi)\big( f_1\big),\dots,   \rho_l'(\phi)\big( f_l'\big)\Big) = \xi(\phi)\big( p(f_1, \dots, f_l')\big). 
\end{equation} 

One could hope not to have to distinguish two types of coordinates, $\prod_{i=1}^k F_i$ and $\prod_{i=1}^l F_i'$, 
by precomposing a sesquilinear map $p$ as above 
with the canonical semilinear bar maps from the complex conjugates of $F_1', \dots, F_l'$ to 
$F_1', \dots, F_l'$; see \cite[Section~3.3]{BT}. 
After such a move, $p$ would become a multilinear map with one type of coordinates; however, after proceeding in this way, 
the equivariance condition 
\eqref{E:peqi} would change in the coordinates $F_1', \dots, F_l'$ and not in the coordinates $F_1, \dots, F_k$ 
keeping the two sets of coordinates distinct.

As in Section~\ref{S:indl}, we assume we are given a boolean action of $L^0$ on the Borel probability space $(X,\gamma)$. 
Also, as in Section~\ref{S:indl}, 
we consider the two unitary representations of $L^0$ associated with the boolean action: 
the Koopman representation $\sigma$ on $H_0(\gamma)$ and 
the representation $\tau$ on $H_\tau$, where $H_\tau$ is the $\ell^2$-sum of the Hilbert spaces $\widetilde{L^2}(\mu^j_x)$, 
with the sum taken over $x\in {\mathbb N}[{\mathbb Z}^\times]$ and $j\in {\mathbb N}$. The two representations 
are isomorphic by the map 
\[
H_\tau \ni f\to {\hat f} \in H_0(\gamma)
\] 
as in \eqref{E:hatt}.

We can now state the result of this section giving a sesquilinear Hilbert space map. 
The lemma below will be used in two special cases, $k=2, l=0$ and $k=0, l=1$, 
but there is little harm in combining both cases into one statement.

\begin{lemma}\label{L:exmul} 
For $1\leq i\leq k$ and $1\leq j\leq l$, let $y_i, z_j\in {\mathbb N}[{\mathbb Z}^\times]$ and 
let $v_i\colon D(y_i)\to 2^n$ and $w_j \colon D(z_j)\to 2^n$ be injections. 
Assume that the images of $v_i$ and $v_j$ are disjoint, for distinct $i,\, j$, as are the images of 
$w_i$, $w_j$, for distinct $i,\, j$, and the images of $v_i$ and $w_j$ for all $i,\, j$. 
Then there exits a sesquilinear Hilbert space map 
\[
p\colon \prod_{i=1}^k L^2(\mu_{y_i} \res \llbracket v_i\rrbracket )\times 
\prod_{i=1}^l L^2(\mu_{z_i} \res \llbracket w_i\rrbracket ) \to H_0(\gamma)
\]
that is equivariant between $\big( (\rho_{y_i})_{i=1}^k, (\rho_{z_i})_{i=1}^l\big)$ on the product 
$\prod_{i=1}^k L^2(\mu_{y_i} \res \llbracket v_i\rrbracket )\times 
\prod_{i=1}^l L^2(\mu_{z_i} \res \llbracket w_i\rrbracket)$ 
and $\sigma$ on $H_0(\gamma)$. 
\end{lemma}

\begin{proof} 
Again, we write out the proof only for $k=l=1$. We set $y=y_1$, $z=z_1$, $v=v_1$, and $w=w_1$. 

For $f\in L^2(\mu_y\res \llbracket v\rrbracket )$, we write ${\hat f}$ for $\hat {\tilde f}$, and similarly 
for $g\in L^2(\mu_z\res \llbracket w\rrbracket )$, we write ${\hat g}$ for $\hat{\tilde g}$, 
where $\tilde f$ and $\tilde g$ are as in Lemma~\ref{L:ident} for $U=\llbracket v\rrbracket$ and $U=\llbracket w\rrbracket$, 
respectively. 

Note that since $v$ and $w$ have disjoint images, 
it follows from Lemma~\ref{L:inde} that, for $f\in L^2(\mu_y\res \llbracket v\rrbracket)$ and $g\in L^2(\mu_z\res \llbracket w\rrbracket)$, 
the functions 
${\hat f}\in H_0(\gamma)$ and ${\hat g}\in H_0(\gamma)$ are independent random variables on $(X,\gamma)$. Therefore, we have 
\[
{\hat f}\, \overline{{\hat g}}\in L^2(\gamma); 
\]
in fact, 
\[
\| {\hat f}\, \overline{{\hat g}}\|^2 = \int_X | {\hat f}|^2 |\overline{{\hat g}}|^2\,d\gamma = \int_X | {\hat f}|^2\, d\gamma\,\int_X |{\hat g}|^2\,d\gamma = 
\| {\hat f}\|^2 \|{\hat g}\|^2. 
\]
Furthermore, 
\[
\int_X {\hat f} \overline{{\hat g}}\,d\gamma = \int_X {\hat f}\, d\gamma\,\int_X \overline{{\hat g}}\,d\gamma =0, 
\]
so 
\begin{equation}\label{E:indfi} 
{\hat f}\, \overline{{\hat g}}\in H_0(\gamma). 
\end{equation}
In view of \eqref{E:indfi}, we consider the map 
\[
p\colon L^2(\mu_y\res \llbracket v\rrbracket )\times L^2\big(\mu_z\res \llbracket w\rrbracket\big)\to H_0(\gamma).
\]
given by 
\[
p(f,g) = {\hat f}\, \overline{\hat{g}}.
\]

The map $p$ is linear in the first coordinate and semilinear in the second one since the maps 
\begin{equation}\label{E:twoma}
L^2(\mu_y\res \llbracket v\rrbracket )\ni f\to {\hat f}\in H_0(\gamma) \; \hbox{ and } \;
L^2\big(\mu_z\res \llbracket w\rrbracket\big)\ni g\to {\hat g}\in H_0(\gamma) 
\end{equation} 
are linear. Further, we claim that 
\[
\big\langle p(f_1,g_1), p(f_2,g_2)\big\rangle = \big\langle f_1, f_2\big\rangle\, \overline{\big\langle g_1, g_2\big\rangle},
\]
for all $f_1, f_2\in L^2(\mu_y\res \llbracket v\rrbracket )$ and $g_1, g_2\in L^2\big(\mu_z\res \llbracket w\rrbracket\big)$. 
To check the above identity, we compute 
\[
\begin{split} 
\big\langle p(f_1,g_1), p(f_2,g_2)\big\rangle &=
\int_X \hat{f_1} \overline{\hat{g_1}}\, \overline{\hat{f_2}\overline{\hat{g_2}}} \,d\gamma
= \int_X \hat{f_1} \overline{\hat{f_2}} \,\overline{ \hat{g_1} \overline{\hat{g_2}}} \,d\gamma\\
&= \int_X \hat{f_1} \overline{\hat{f_2}} \, d\gamma \int_X\overline{ \hat{g_1} \overline{\hat{g_2}}}\,d\gamma
= \big\langle f_1, f_2\big\rangle\, \overline{\big\langle g_1, g_2 \big\rangle}, 
\end{split} 
\]
where the finiteness of the first integral follows from \eqref{E:indfi}. 
The fourth equality holds since the maps \eqref{E:twoma} 
are Hilbert space maps. 
To justify the third equality, note that the functions 
$\hat{f_1} \overline{\hat{f_2}}$ and $\overline{\hat{g_1} \overline{\hat{g_2}}}$ 
are in $L^1(\gamma)$ and they are independent random variables on $(X,\gamma)$ by Lemma~\ref{L:inde} since $v$ and $w$ 
have disjoint images. Since the second equality is obvious, we proved that $p$ is a sesquilinear Hilbert space map. 

To check equivariance, note that since $\sigma$ is a Koopman representation, we have
\[
\sigma(\phi)\big( {\hat f}\,\overline{{\hat g}}\big) = \sigma(\phi)\big( {\hat f}\big)\, \overline{\sigma(\phi)\big({\hat g}\big)},
\]
for all $\phi\in L^0$ and $f\in L^2(\mu_y\res \llbracket v\rrbracket )$, and $g\in L^2\big(\mu_z\res \llbracket w\rrbracket\big)$. 
Thus, 
\[
\sigma(\phi)\big(p(f,g)\big) = \sigma(\phi)\big( {\hat f}\big)\, \overline{\sigma(\phi)\big({\hat g}\big)} = \widehat{ \rho_y(\phi)\big( f\big)}\, 
\overline{\widehat{ \rho_z(\phi)\big( g\big)}} = p\big( \rho_y(\phi)\big( f\big), \rho_z(\phi)\big( g\big)\big),
\]
where the second equality holds by the equivariance of the maps in \eqref{E:twoma}. The equivariance of $p$ follows from 
the above equation. 
\end{proof}

\section{A lemma on tensor products}\label{S:tens}

We will need the general Lemma~\ref{L:tenor} below that identifies the tensor product of  
$L^2(\mu_1), \dots, L^2(\mu_k)$ and $L^2(\nu_1), \dots , L^2(\nu_l)$ as $L^2(\prod_{i=1}^k \mu_i\times \prod_{I=1}^l \nu_i)$, 
for finite Borel measures $\mu_1, \dots, \mu_k$ and $\nu_1, \dots, \nu_l$, in a category we are interested in. It is a consequence 
of the standard development of the tensor product of Hilbert spaces as in \cite[Appendix E]{Jan}. The two 
special cases of it that will be used are $k=2, l=0$ and $k=0, l=1$. 

Define the canonical map 
\begin{equation}\label{E:seca}
q\colon \prod_{i=1}^kL^2(\mu_i)\times \prod_{i=1}^l L^2(\nu_i)\to L^2\big(\prod_{i=1}^k \mu_i\times \prod_{i=1}^l \nu_i\big),
\end{equation} 
by letting, for $f_i\in L^2(\mu_i)$, for $1\leq i\leq k$, and $g_i\in L^2(\nu_i)$, for $1\leq i\leq l$, 
\[
q\big(f_1, \dots ,g_l\big)(\alpha_1, \dots, \beta_l)= \prod_{i=1}^k f_i(\alpha_i)\prod_{i=1}^l \overline{g_i(\beta_i)}. 
\]
Then it is routine to check that $q$ is a sesquilinear Hilbert space map.

\begin{lemma}\label{L:tenor} 
Let $\mu_i$, for $1\leq i\leq k$, and $\nu_i$, for $1\leq i\leq l$, be finite Borel measures. 
Let $H$ be a Hilbert space whose inner product is $\langle\cdot, \cdot\rangle_H$. Then,  
for each sesquilinear Hilbert space map,  
\[
p\colon \prod_{i=1}^kL^2(\mu_i)\times \prod_{i=1}^l L^2(\nu_i) \to H,
\]
there exists a Hilbert space map  
\[
r\colon  L^2\big(\prod_{i=1}^k \mu_i\times \prod_{i=1}^l \nu_i\big)  \to H
\]
with 
\[
p=r\circ q.
\]
\end{lemma} 

\begin{proof} We write out the proof for $k=l=1$ only. We set $\mu=\mu_1$ and $\nu=\nu_1$. 
We denote the inner products in the spaces 
$L^2(\mu)$, $L^2(\nu)$, $L^2(\mu\times \nu)$
by 
\[
\langle \cdot, \cdot\rangle_\mu,\; \langle \cdot, \cdot\rangle_\nu, \; \langle \cdot, \cdot\rangle_{\mu\times\nu},
\]
respectively.

We precompose $p$ and $q$ with the bijection 
\begin{equation}\label{E:seco}
L^2(\mu)\times L^2(\nu)\ni (f,g) \to (f, \overline{g}) \in L^2(\mu)\times L^2(\nu), 
\end{equation} 
which is linear in the fist coordinate and semilinear in the second one. 
As a result we obtain bilinear maps $p'$ and $q'$. 
Let $L$ be the linear span in $L^2(\mu\times \nu)$ 
of the image of $q'$. We note that $L$ is a dense subspace of $L^2(\mu\times \nu)$. 
By \cite[Definitions~E1, E7 and Examples~E6, E10]{Jan}, there exists a linear function $r\colon L\to H$ such that 
\begin{equation}\label{E:agree}
p'=r\circ q'.
\end{equation}
Using identity \eqref{E:agree} and the fact that $p'$ and $q'$ are bilinear Hilbert space maps, we calculate, for $f_1, f_2\in L^2(\mu)$ 
and $g_1, g_2\in L^2(\nu)$, 
\[
\begin{split}
\big\langle r\big(q'(f_1, g_1)\big), r\big(q'(f_2,g_2)\big)\big\rangle_H 
&=\big\langle p'(f_1, g_1), p'(f_2,g_2)\big\rangle_H\\ 
&= \langle f_1, f_2\rangle_\mu \langle g_1, g_2\rangle_\nu 
=\big\langle q'(f_1, g_1), q'(f_2,g_2)\big\rangle_{\mu\times\nu}.
\end{split}
\]
By linearity, the above identity extends to 
\begin{equation}\label{E:hilma}
\big\langle r(e_1), r(e_2)\big\rangle_H = \langle e_1, e_2\rangle_{\mu\times\nu}, \hbox{ for all }e_1, e_2\in L. 
\end{equation} 
This, in particular, means that $r$ is an isometry and, therefore, by density of $L$, 
it extends to a continuous linear map defined on the whole space 
$L^2(\mu\times \nu)$. We denote this extension again by $r$, so we still have \eqref{E:agree}, which, by the bijectivity 
of the map in \eqref{E:seco}, gives 
\[
p=r\circ q. 
\]
By continuity of $r$ and density of $L$, we see that \eqref{E:hilma} holds for all $e_1, e_2\in L^2(\mu\times \nu)$, as required. 
\end{proof}

\section{Two representations of $L^0(\lambda, {\mathbb T})$ and a Hilbert space map}\label{S:more}

In addition to the unitary representations $\sigma$ and $\tau$ defined in Section~\ref{S:indl}, 
the following unitary representation will be used. 

{\bf Representations $\rho_{U,{\bar \iota}}$ and $\rho_{U,e}$.} 
Let $y, z\in {\mathbb N}[{\mathbb Z}^\times]$. Let $U$ be a basic set for $y\oplus z$ and let 
$\bar\iota$ be as in \eqref{E:seq} for $y$ and $z$. 
Then, $\bar\iota$ gives rise to the homeomorphism 
\begin{equation}\notag
h_{\bar\iota}\colon C_y\times C_{z} \to  C_{y\oplus z}.
\end{equation}
as in \eqref{E:hom}. 
Let 
\begin{equation}\label{E:meh}
\mu_{U,{\bar\iota}} = \big(\big(h_{\bar\iota}\big)_*(\mu_y\times \mu_z)\big)\res U. 
\end{equation}
By Lemma~\ref{L:compa}\,(i),  
$\mu_{U,{\bar\iota}}$ is a finite measure on $C^0_{y\oplus z}$ that is marginally compatible with $y\oplus z$. 
This observation allows us to consider the representation $\rho_{U, {\bar\iota}}$ equal to $\rho_{y\oplus z}$ on 
$L^2(\mu_{U,{\bar\iota}})$, that is,  $\rho_{U, {\bar\iota}}$ is given by 
\begin{equation}\notag
L^0\times L^2(\mu_{U, {\bar\iota}}) \ni (\phi, f) 
\to\rho_{y\oplus z}(\phi)\big( f\big) \in L^2(\mu_{U, {\bar\iota}}). 
\end{equation}

Let $z\in {\mathbb N}[{\mathbb Z}^\times]$. Let $U$ be a basic set for $(-1)z$, and let 
\begin{equation}\notag
e\colon C_z\to  C_{(-1)z}, 
\end{equation}
be the homeomorphism as in \eqref{E:hom2} with $\ell=-1$. Put 
\begin{equation}\label{E:mee}
\mu_{U,e} = e_*(\mu_z)\res U. 
\end{equation}
By Lemma~\ref{L:compa2},  
$\mu_{U,{\bar\iota}}$ is a finite measure on $C^0_{(-1)z}$ that is marginally compatible with $(-1)z$. 
As above, we can now consider the representation $\rho_{U, e}$ equal to $\rho_{(-1)z}$ on 
$L^2(\mu_{U,{\bar\iota}})$, that is, 
\begin{equation}\notag
L^0\times L^2(\mu_{U, e}) \ni (\phi, f) 
\to\rho_{(-1)z}(\phi)\big( f\big) \in L^2(\mu_{U, e}). 
\end{equation}

The following lemma gives the desired connections between the representations $\rho_{U,{\bar\iota}}$ 
$\rho_{U,e}$
and $\sigma$. It uses the work from Sections~\ref{S:indl}, \ref{S:sesq}, and \ref{S:tens}.

\begin{lemma}\label{L:emps} 
\begin{enumerate} 
\item[(i)] Let $U$ and $\bar\iota$ be as in the definition of $\rho_{U,{\bar\iota}}$. There is a Hilbert space map  
\begin{equation}\notag
\Psi_{U, {\bar\iota}}: L^2(\mu_{U, {\bar\iota}})\to H_0(\gamma) 
\end{equation} 
that is equivariant between $\rho_{U, {\bar\iota}}$ and $\sigma$. 

\item[(ii)]  Let $U$ and $e$ be as in the definition of $\rho_{U,e}$. There is a Hilbert space map  
\begin{equation}\notag
\Psi_{U,e} : L^2(\mu_{U, e})\to H_0(\gamma) 
\end{equation} 
that is equivariant between $\rho_{U, e}$ and $\sigma$. 
\end{enumerate}
\end{lemma}

\begin{proof} (i) We need to set up some notation. 
We have ${\bar\iota} = (\iota^y, \iota^{z})$ for some 
\[
\iota^y\colon D(y)\to D(y\oplus z)\;\hbox{ and }\;\iota^{z}\colon D(z)\to D(y\oplus z).
\]
Let $U = \llbracket u\rrbracket$ for some injection $u\colon D(y\oplus z)\to 2^n$.
Define injections 
$v\colon D(y)\to 2^n$ and $w\colon D(z)\to 2^n$ by letting 
\[
v= u\circ \iota^y, \;\hbox{ and }\; w= u\circ\iota^z. 
\]
By Lemma~\ref{L:intp}, we have 
\begin{equation}\label{E:uvw1}
h_{\bar\iota}(\llbracket v\rrbracket\times \llbracket w\rrbracket) = U. 
\end{equation}

The canonical bilinear Hilbert space map, as in \eqref{E:seca} with $k=2$ and $l=0$, 
\[
q\colon L^2\big(\mu_y\res \llbracket v\rrbracket\big)\times L^2\big(\mu_z\res \llbracket w\rrbracket\big)\to  
L^2\big(\mu_y\res \llbracket v\rrbracket  \times \mu_z\res \llbracket w\rrbracket\big)
\]
is given by 
\[
q\big(f, g\big)(\alpha, \beta)=f(\alpha)g(\beta). 
\]
The following claim gives the relevant to us equivariance property of $q$. 

\begin{claim}\label{C:isrp}
There is a Hilbert space isomorphism 
\[
r'\colon L^2(\mu_{U,{\bar\iota}}) \to L^2\big(\mu_y\res \llbracket v\rrbracket \times \mu_z\res \llbracket w\rrbracket\big)
\]
such that the bilinear Hilbert map 
\[
(r')^{-1}\circ q\colon  L^2(\mu_y\res \llbracket v\rrbracket)\times L^2(\mu_z\res \llbracket w\rrbracket)   \to L^2(\mu_{U,{\bar\iota}})
\]
is equivariant between $(\rho_y, \rho_z)$ on $L^2(\mu_y\res \llbracket v\rrbracket)\times L^2(\mu_z\res \llbracket w\rrbracket)$ and 
$\rho_{U, {\bar\iota}}$ on $L^2(\mu_{U,{\bar\iota}})$. 
\end{claim} 

\noindent {\em Proof of Claim~\ref{C:isrp}.} For $f\in  L^2(\mu_{U,{\bar\iota}})$, we define 
\[
r'(f) = f\circ h_{\bar\iota}.
\]
It is clear from \eqref{E:uvw1} and from the definition of $\mu_{U,{\bar\iota}}$ that $r'$ is a Hilbert space isomorphism. 

It remains to check the equivariance condition on $(r')^{-1}\circ q$, that is, for $\phi\in L^0$ and functions 
$f\in  L^2(\mu_y\res \llbracket v\rrbracket)$ and $g\in L^2(\mu_z\res \llbracket w\rrbracket)$, we need to see that  
\[
q\Big(\rho_y(\phi)\big(f\big), \rho_z(\phi)\big(g\big)\Big)\circ h_{\bar\iota}^{-1}  = 
\rho_{U,{\bar\iota}}(\phi)\big(q(f,g)\circ h_{\bar\iota}^{-1}\big).
\]
Checking the condition above amounts to showing that 
\[
q\big(R_y(\phi) f,\, R_z(\phi) g\big) 
= \big( R_{y\oplus z} (\phi)  \circ h_{\bar\iota})\big)\, q(f,g), 
\]
which is a restatement of Lemma~\ref{L:idio}\,(i). The claim is proved. 

\smallskip

Note that $v$ and $w$ have disjoint images since $\iota^y$ and $\iota^z$ have disjoint images and $u$ is injective. 
By Lemma~\ref{L:exmul}, these conditions guarantee the existence of 
a bilinear Hilbert space map 
\[
p\colon L^2(\mu_y\res \llbracket v\rrbracket )\times L^2\big(\mu_z\res \llbracket w\rrbracket\big)\to H_0(\gamma).
\]
By Lemma~\ref{L:tenor}, $p$ factors through the canonical bilinear map $q$. 
The factorization 
produces a Hilbert space map  
\[
r\colon L^2\big(\mu_y\res \llbracket v\rrbracket \times \mu_z\res \llbracket w\rrbracket\big)\to  H_0(\gamma)
\]
so that $p=r\circ q$. 

Let 
\[
r'\colon L^2(\mu_{U,{\bar\iota}}) \to L^2\big(\mu_y\res \llbracket v\rrbracket \times \mu_z\res \llbracket w\rrbracket\big)
\]
be the Hilbert space isomorphism from Claim~\ref{C:isrp}. Define 
\[
\Psi_{U, {\bar\iota}} = r\circ r'\colon L^2(\mu_{U, {\bar\iota}})\to H_0(\gamma). 
\]
Clearly, this is a Hilbert space map. 

It remains to show that $\Psi_{U, {\bar\iota}}$ is equivariant between $\rho_{U, {\bar\iota}}$ and $\sigma$. Set 
\[
q'= (r')^{-1}\circ q. 
\]
Then, by Claim~\ref{C:isrp}, $q'$ is an equivariant bilinear Hilbert space map and obviously 
\[
p= \Psi_{U, {\bar\iota}}\circ q'. 
\]
Now the equivariance of $\Psi_{U, {\bar\iota}}$ follows from this equation, the equivariance of $p$ and $q'$, and the density of the linear 
span of the image of $q'$, which is equal to the image of $q$, in the space 
$L^2\big(\mu_y\res \llbracket v\rrbracket \times \mu_z\res \llbracket w\rrbracket\big)$.

(ii) This is similar to (i). 
Let $U = \llbracket u\rrbracket$ for some injection $u\colon D((-1) z)\to 2^n$.
Define an injection $w\colon D(z)\to 2^n$ by letting $w(k,i)= u(-k,i)$. 
An easy calculation shows that 
\begin{equation}\label{E:gagar}
e \big(\llbracket w\rrbracket \big) =U.
\end{equation}

The canonical semilinear Hilbert space map, as in \eqref{E:seca} with $k=0$ and $l=1$, 
\[
q\colon L^2\big(\mu_z\res \llbracket w\rrbracket\big)\to  
L^2\big( \mu_z\res \llbracket w\rrbracket\big)
\]
is $q(f) =\overline{f}$. The following claim gives the equivariance property of $q$.

\begin{claim}\label{C:isrp2}
There is a Hilbert space isomorphism 
\[
r'\colon L^2(\mu_{U,e}) \to L^2\big(\mu_z\res \llbracket w\rrbracket\big)
\]
such that the Hilbert map 
\[
(r')^{-1}\circ q\colon  L^2(\mu_z\res \llbracket w\rrbracket)   \to L^2(\mu_{U,e})
\]
is equivariant between $\rho_z$ on $L^2(\mu_z\res \llbracket w\rrbracket)$ and $\rho_{U,e}$ on $L^2(\mu_{U,e})$. 
\end{claim} 

\noindent {\em Proof of Claim~\ref{C:isrp2}.} For $f\in  L^2(\mu_{U,{\bar\iota}})$, we define 
\[
r'(f) = f\circ e.
\]
From \eqref{E:gagar} and from the definition of $\mu_{U,e}$, $r'$ is a Hilbert space isomorphism. 

To check the equivariance of $(r')^{-1}\circ q$, we have to show that, for $\phi\in L^0$ and a function 
$g\in L^2(\mu_z\res \llbracket w\rrbracket)$, 
\[
q\big(\rho_z(\phi)(g)\big)\circ e^{-1}  = \rho_{U, e}(\phi)\big(q(g)\circ e^{-1}\big).
\]
This condition is equivalent to 
\[
\overline{R_z(\phi) g} 
= \big( R_{(-1)z}  \circ e\big)\, \overline{g},  
\]
which is Lemma~\ref{L:idio}\,(ii) with $\ell=-1$ since $\overline{R_z(\phi)} = R_z(\phi)^{-1}$. The claim follows.  

\smallskip 

By Lemma~\ref{L:exmul}, there exists 
a semilinear Hilbert space map 
\[
p\colon L^2\big(\mu_z\res \llbracket w\rrbracket\big)\to H_0(\gamma).
\]
By Lemma~\ref{L:tenor}, $p$ factors through the canonical map $q$. 
The factorization produces a Hilbert space map 
\[
r\colon L^2\big( \mu_z\res  \llbracket w \rrbracket\big)\to  H_0(\gamma)
\]
so that $p=r\circ q$. 
We take now $r'$ from Claim~\ref{C:isrp2} and define $\Psi_{U,e} = r\circ r'$. 
This is a Hilbert space map. To see the equivariance of $\Psi_{U,e}$ between $\rho_{U, e}$ and $\sigma$, set 
\[
q'= (r')^{-1}\circ q. 
\]
Then, by Claim~\ref{C:isrp2}, $q'$ is an equivariant semilinear Hilbert space map and obviously 
$p= \Psi_{U,e}\circ q'$.
The equivariance of $\Psi_{U,e}$ follows from this equation, the equivariance of $p$ and $q'$, and the 
surjectivity of $q'$. 
\end{proof}

\section{Proof of Theorem~\ref{T:notL}}\label{S:finp}

By Lemmas~\ref{L:abseio}\,(ii) and \ref{L:abse}\,(ii), the general case reduces to proving 
\begin{equation}\label{E:plen} 
\mu_y\otimes \mu_z \preceq \mu_{y\oplus z}
\end{equation}
and 
\begin{equation}\label{E:taen} 
(-1)\mu_z\preceq \mu_{(-1)z}. 
\end{equation}

We show \eqref{E:plen} first. Set $x= y\oplus z$. Note that, for each set $U$ basic for $x$, we have 
\begin{equation}\label{E:alml}
\big(\mu_y\otimes \mu_z\big)\res U = \sum_{\bar\iota} \mu_{U,{\bar\iota}},
\end{equation} 
where $\bar\iota$ ranges over all pairs as in \eqref{E:seq} for $y$ and $z$ and $\mu_{U,{\bar\iota}}$ is as in \eqref{E:meh}. 
Recall that basic sets for $x$ form a basis of $C^0_x$ by Lemma~\ref{L:basi} and $\mu_y\otimes \mu_z$ is concentrated in 
$C^0_x$ by Lemma~\ref{L:compa}\,(ii). So, by \eqref{E:alml}, 
to see \eqref{E:plen}, it suffices to show that, for each set $U$ basic for $x$ and each $\bar\iota$ as above, 
\[
\mu_{U,{\bar\iota}}\preceq \mu_x. 
\]
It will be enough to see that, for each compact set $K\subseteq U$, if $\mu_{U,{\bar\iota}}(K)>0$, then $\mu_x(K)>0$. 

Fix a compact set $K\subseteq U$ with $\mu_{U,{\bar\iota}} (K)>0$. Consider 
\[
\Gamma = \Phi\circ \Psi_{U, {\bar\iota}} \colon L^2(\mu_{U,{\bar\iota}})\to H_\tau, 
\]
where $\Phi$ and $\Psi_{U, {\bar\iota}}$ are as in \eqref{E:hot} and Lemma~\ref{L:emps}\,(i), respectively. 
Then $\Gamma$ is a Hilbert space map  
and, by Lemma~\ref{L:kpr} in combination with Lemma~\ref{L:co1}, we have 
\begin{equation}\label{E:prn}
\Gamma\big( L^2(\mu_{U,{\bar\iota}}\res K)\big)\subseteq [x,K]^\tau. 
\end{equation} 
Since $\mu_{U, {\bar\iota}}(K)>0$, the space $L^2(\mu_{U,{\bar\iota}}\res K)$ is nontrivial. Therefore, since $\Gamma$, being 
a Hilbert space map, is an embedding, 
$\Gamma\big( L^2(\mu_{U, {\bar\iota}}\res K)\big)$ 
is non-trivial, which, by \eqref{E:prn}, makes the space $[x,K]^\tau$ non-trivial. Now, it follows, 
by Lemma~\ref{L:co2}, that there exits $j$, for which the space 
$\widetilde{L^2}(\mu^j_x\res \widetilde{K})$ is non-trivial, that is, $\mu^j_x(\widetilde{K})>0$. (Recall 
the definition of $\widetilde{K}$ from \eqref{E:goar}.) Since $\mu^j_x$, being compatible with $x$,  
is invariant under all good homeomorphisms of $C_x$, this last inequality gives $\mu^j_x(K)>0$. 
Since, for each $j$, $\mu^j_x\preceq \mu_x$, we get $\mu_x(K)>0$, as required. 

The proof of \eqref{E:taen} is similar. Since, for each set $U$ basic for $(-1)z$, we have 
\begin{equation}\notag
(-1)\mu_z\res U = \mu_{U,e},
\end{equation} 
where $\mu_{U,e}$ is as in \eqref{E:mee}, it suffices 
to see that, for each compact set $K\subseteq U$, if $\mu_{U,e}(K)>0$, then $\mu_{(-1)z}(K)>0$.
We fix a compact set $K\subseteq U$ with $\mu_{U,e} (K)>0$ and consider 
\[
\Gamma = \Phi\circ \Psi_{U, e} \colon L^2(\mu_{U,e})\to H_\tau, 
\]
where $\Phi$ and $\Psi_{U, e}$ are as in \eqref{E:hot} and Lemma~\ref{L:emps}\,(ii), respectively. 
Then $\Gamma$ is a Hilbert space map, for which, by combining Lemmas~\ref{L:kpr} and \ref{L:co1}, 
we have 
\begin{equation}\notag
\Gamma\big( L^2(\mu_{U,e}\res K)\big)\subseteq [(-1)z,K]^\tau. 
\end{equation} 
The proof is now finished as in the case of \eqref{E:plen}. 

\bigskip

\noindent {\bf Acknowledgement.} I thank Justin Moore and Joe Rosenblatt for stimulating conversations and email exchanges 
on the subject matter of the paper. I am also grateful to Eli Glasner and Benjy Weiss for pointing out the martingale proof of Theorem~\ref{T:ert}  and 
to Eli Glasner for clarifications concerning the literature.

\end{document}